\documentclass[11pt,  reqno]{amsart}

\usepackage{amsmath,amssymb,amscd,amsthm,amsxtra, esint,bbm}

\usepackage{mathrsfs} 
\usepackage{enumerate}

\usepackage[implicit=true]{hyperref}

\usepackage{color}
\usepackage{ulem}
\usepackage[makeroom]{cancel}

\usepackage{url}

\usepackage[left=32mm, right=32mm, 
bottom=27mm]{geometry}

\allowdisplaybreaks[2]

\definecolor{gr}{rgb}   {0.,   0.69,   0.23 }
\definecolor{bl}{rgb}   {0.,   0.5,   1. }
\definecolor{mg}{rgb}   {0.85,  0.,    0.85}
\definecolor{yl}{rgb}   {0.8,  0.7,   0.}
\definecolor{or}{rgb}  {0.7,0.2,0.2}

\newtheorem{theorem}{Theorem} [section]

\newtheorem{lemma}[theorem]{Lemma}
\newtheorem{proposition}[theorem]{Proposition}
\newtheorem{remark}[theorem]{Remark}

\DeclareMathOperator*{\intt}{\int}

\newcommand{\noi}{\noindent}

\newcommand{\R}{\mathbb{R}}

\newcommand{\T}{\mathbb{T}}

\let\Re=\undefined\DeclareMathOperator*{\Re}{Re}
\let\Im=\undefined\DeclareMathOperator*{\Im}{Im}

\let\P= \undefined
\newcommand{\P}{\mathbf{P}}

\newcommand{\E}{\mathbb{E}}

\newcommand{\F}{\mathcal{F}}

\newcommand{\be}{\beta}
\newcommand{\dl}{\delta}

\newcommand{\nb}{\nabla}

\newcommand{\Dl}{\Delta}
\newcommand{\eps}{\varepsilon}

\newcommand{\g}{\gamma}

\newcommand{\ld}{\lambda}

\newcommand{\Si}{\Sigma}
\newcommand{\ft}{\widehat}

\newcommand{\cj}{\overline}

\newcommand{\dt}{\partial_t}

\newcommand{\HS}{{\textup{HS}}}
\newcommand{\HSi}{{\textup{HS}(L^2;L^2)}}
\newcommand{\HSii}{{\textup{HS}(L^2;\dot{H}^1)}}
\newcommand{\hi}{{\textup{hi}}}
\newcommand{\lo}{{\textup{low}}}

\newcommand{\supT}{\sup_{0\le t\le \tau}}

\newcommand{\uu}{u^\ld}
\newcommand{\ta}{\theta}

\renewcommand{\o}{\omega}
\renewcommand{\O}{\Omega}

\newcommand{\les}{\lesssim}
\newcommand{\ges}{\gtrsim}

\newcommand{\jb}[1]
{\langle #1 \rangle}

\renewcommand{\b}{\be}
\newcommand{\ind}{\mathbf 1}

\newcommand{\N}{\mathbb{N}}

\newcommand{\NN}{\mathcal{N}}

\newtheorem*{ackno}{Acknowledgements}

\numberwithin{equation}{section}
\numberwithin{theorem}{section}

\begin{document}
\baselineskip = 14pt

\title[Almost conservation laws for stochastic NLS]
{Almost conservation laws for stochastic nonlinear Schr\"odinger equations}
\author[ K.~Cheung,  G.~Li, and T.~Oh]
{Kelvin Cheung, Guopeng Li, and Tadahiro Oh}

\address{
Kelvin Cheung\\
Department of Mathematics\\ Heriot-Watt University\\ and 
the Maxwell Institute for the Mathematical Sciences\\
Edinburgh\\ EH14 4AS\\ United Kingdom}

\email{K.K.Cheung-3@sms.ed.ac.uk}

\address{
Guopeng Li\\
School of Mathematics\\
The University of Edinburgh\\
and The Maxwell Institute for the Mathematical Sciences\\
James Clerk Maxwell Building\\
The King's Buildings\\
 Peter Guthrie Tait Road\\
Edinburgh\\ 
EH9 3FD\\United Kingdom} 

\email{G.Li-18@ed.ac.uk}

\address{
Tadahiro Oh\\
School of Mathematics\\
The University of Edinburgh\\
and The Maxwell Institute for the Mathematical Sciences\\
James Clerk Maxwell Building\\
The King's Buildings\\
 Peter Guthrie Tait Road\\
Edinburgh\\ 
EH9 3FD\\United Kingdom} 

\email{hiro.oh@ed.ac.uk}

\begin{abstract}
In this paper, we present a globalization argument 
for stochastic nonlinear dispersive PDEs with additive noises by adapting the $I$-method 
(= the method of almost conservation laws) to the stochastic setting.
As a model example, we  consider the defocusing stochastic cubic nonlinear Schr\"odinger equation (SNLS)
on $\R^3$
with additive stochastic forcing, white in time and correlated in space,
such that the noise lies below the energy space.
By  combining the $I$-method  with Ito's lemma
and a stopping time argument, 
we construct global-in-time dynamics for SNLS below the energy space.
\end{abstract}




\maketitle


\section{Introduction}
\subsection{Stochastic nonlinear Schr\"odinger equation}

We consider the Cauchy problem for the  stochastic nonlinear Schr\"{o}dinger equation (SNLS)
with an additive noise:
\begin{equation}
\label{SNLS}
\begin{cases}
 i \dt u + \Dl  u = |u|^{p-1} u + \phi \xi \\
 u|_{t = 0} = u_0,
\end{cases}
\qquad (t, x) \in \R\times \R^d,
\end{equation}

\noi
where $\xi (t,x)$ denotes a (Gaussian) space-time white noise on $\R \times \R^d$ 
and $\phi$ is a bounded operator on $L^2(\R^d)$. 
In this paper, we restrict our attention to the defocusing case.
Our main goal is to establish global well-posedness of \eqref{SNLS} 
in the energy-subcritical case
with a rough noise,
namely,  
with 
a noise  not belonging to the energy space $H^1(\R^d)$.
Here, the energy-subcriticality refers to
the following range of $p$:   
(i)~$1 < p < 1 + \frac{4}{d-2}$ for  $d \geq 3$
and (ii)~$1 < p < \infty$ for $d = 1, 2$.
In terms of the scaling-critical regularity   $s_\text{crit}$ defined by 
 \begin{align*}
 s_\text{crit} = \frac d2 - \frac{ 2}{p-1}, 
 \end{align*}

\noi
the energy-subcriticality is equivalent to the condition $ s_\text{crit} < 1$.

We say that $u$ is a solution to \eqref{SNLS} on a given time interval $[-T, T]$
if it satisfies the following 
Duhamel formulation (= mild formulation):
\begin{align}
u(t) = S(t) u_0 -i \int_0^t S(t-t') |u|^{p-1}u(t') dt' -i \int_0^t S(t-t') \phi \xi (dt')
\label{SNLS2}
\end{align}

\noi
in $C([-T, T]; B(\R^d))$, 
where  $S(t) = e^{it \Dl}$ denotes the linear Schr\"{o}dinger propagator
and $B(\R^d)$ is a suitable Banach space of functions on $\R^d$.
In this paper, we take $B(\R^d)$ to be the $L^2$-based Sobolev space $H^s(\R^d)$
for some suitable $s \in \R$.
We say that $u$ is a global solution to~\eqref{SNLS}
if \eqref{SNLS2} holds  
in $C([-T, T]; B(\R^d))$
for any $T> 0$. 
We often construct a solution~$u$ 
belonging to 
 $C([-T, T]; B(\R^d))\cap X([-T, T])$, where  
$X([-T, T])$ denotes some auxiliary function space
such as the Strichartz spaces $L^q([-T, T]; W^{s, r}(\R^d))$; see \cite{DD2, OPW}.
For our purpose, we take 
this auxiliary function space $X([-T, T])$ to be 
(local-in-time version of) the Fourier restriction norm space
(namely, the $X^{s, b}$-space defined in~\eqref{Xsb} below).

The last term on the right-hand side of \eqref{SNLS2} represents the effect of the stochastic forcing 
and is called the stochastic convolution, which we denote by $\Psi$:
\begin{align}
\Psi(t) =  - i \int_0^t S(t - t') \phi \xi (dt').
\label{SNLS3}
\end{align}

\noi
See Subsection \ref{SUBSEC:Stco} for the precise meaning of 
the definition \eqref{SNLS3}; see \eqref{Psi1} and \eqref{Psi2}.
In the following, we assume that $\phi \in \HS(L^2; H^s)$
for appropriate values of $s \geq 0$, 
namely, $\phi$ is taken to be a Hilbert-Schmidt
operator from $L^2(\R^d)$ to $H^s(\R^d)$.
It is easy to see  that  $\phi \in \HS(L^2; H^s)$
implies
$\Psi \in C(\R; H^s(\R^d))$ almost surely; see \cite{DZ}.
Our main interest is to study \eqref{SNLS}
when $\phi \in \HS(L^2; H^s)$ for $s < 1$
such that the stochastic convolution does not belong to the energy space $H^1(\R^d)$.

When $\phi = 0$, 
the equation \eqref{SNLS} reduces to 
the  (deterministic)  defocusing nonlinear Schr\"odinger equation~(NLS):
\begin{equation}
i \partial_t u +   \Delta u =  |u|^{p-1} u.  
\label{NLS1}
\end{equation}

\noi
A standard contraction argument with the Strichartz estimates
 (see \eqref{Str1} below)
yields
local well-posedness of 
\eqref{NLS1}  in $H^s(\R^d)$
when $s \geq \max(s_\text{crit}, 0)$; 
see 
\cite{GV79,  Kato, Tsu, CW}.\footnote
{When $p$ is not an odd integer, 
we may need to impose an extra assumption 
due to the non-smoothness of the nonlinearity.
A similar comment applies to the case of SNLS.
}
On the other hand, 
 \eqref{NLS1} is known to be ill-posed
in the scaling supercritical regime:  $s < s_\text{crit}$.
See \cite{CCT, Kishimoto,  O17}.
In the energy-subcritical case, 
 global well-posedness of \eqref{NLS1} in $H^1(\R^d)$
easily follows
from iterating the local-in-time argument
in view of 
the 
following conservation laws
for~\eqref{NLS1}:
\begin{align}
\begin{split}
\text{Mass: }&  M(u(t)) = \int_{\R^d} |u(t, x)|^2 dx,\\
\text{Energy: } &  E(u(t)) = \frac 12 \int_{\R^d} |\nb u(t, x)|^2 dx
+ \frac{1}{p+1} \int_{\R^d} |u(t, x)|^{p+1} dx,
\end{split}
\label{cons}
\end{align}

\noi
providing a  global-in-time a priori control on the $H^1$-norm of a solution
to \eqref{NLS1}.

There are analogues of these well-posedness results
in the context of  SNLS \eqref{SNLS}. 
In~\cite{DD2}, de Bouard and Debussche 
studied \eqref{SNLS}
in the energy-subcritical setting, 
assuming that $\phi \in \HS(L^2;H^1)$.
By using  the Strichartz estimates, 
they showed that  the stochastic convolution $\Psi$
almost surely belongs to a right Strichartz space, 
which  allowed them to prove
local well-posedness of \eqref{SNLS} in $H^1(\R^d)$. 
When  $s \geq \max (s_\text{crit}, 0)$,
a slight modification of  the argument in \cite{DD2}
and the improved space-time regularity  of the stochastic convolution 
(see Lemma~\ref{LEM:stoconv} below)
yields
local well-posedness
 of \eqref{SNLS}
in $H^s(\R^d)$, 
provided that $\phi\in \HS(L^2; H^s)$.
In the energy-subcritical case, 
one can adapt the globalization argument
for the deterministic NLS \eqref{NLS1}, 
based on the conservation laws~\eqref{cons}, 
to the stochastic setting
with a sufficiently regular noise.
More precisely, 
assuming $\phi \in \HS(L^2; H^1)$, 
de Bouard and Debussche~\cite{DD2}
proved global well-posedness of~\eqref{SNLS} in $H^1(\R^d)$
by 
applying Ito's lemma
to the mass $M(u)$ and the energy $E(u)$ in \eqref{cons}
and 
establishing
an  a priori $H^1$-bound of solutions to \eqref{SNLS}. 
In this paper, 
we  also consider 
the energy-subcritical case
but we treat a rougher noise: 
 $\phi\in \HS(L^2; H^s)$ for $s < 1$.

In the deterministic setting, 
Colliander, Keel, Staffilani, Takaoka, and Tao \cite{CKSTT0}
introduced
the so-called $I$-method (also known as the method of almost conservation laws)
and proved global well-posedness
of the energy-subcritical defocusing cubic NLS (\eqref{NLS1} with $p = 3$) on $\R^d$, $d = 2, 3$,
below the energy space.
Since then, 
the $I$-method has been applied to a wide class of  dispersive models
in establishing  global well-posedness below the energy spaces
(or more generally below regularities associated 
with conservation laws), 
where there is no a priori bound
on relevant norms 
(for iterating a local-in-time argument)
directly given by a conservation law.
Our strategy 
for proving global well-posedness of SNLS \eqref{SNLS}
when  $\phi\in \HS(L^2; H^s)$, $s < 1$, 
is to implement the $I$-method
in the stochastic PDE setting.
This will provide a general framework 
for establishing global well-posedness
of stochastic dispersive equations
with additive noises 
below energy spaces.

\subsection{Main result}

For the sake of concreteness, 
we consider SNLS \eqref{SNLS} in the three-dimensional cubic case
($d = 3$ and $p = 3$):
\begin{equation}
 \begin{cases}
i \dt  u + \Dl u =|u|^2 u +\phi \xi    \\ 
 u|_{t=0}=u_0\in H^s(\R^3),
 \end{cases}
 \qquad (t, x) \in \R\times \R^3.
\label{SNLS0}
 \end{equation}

\noi
We point out,  however,  that our implementation
of the $I$-method in the stochastic PDE setting
is sufficiently general 
and can be easily adapted  to other dispersive models with rough additive stochastic forcing. We now state our main result.

\begin{theorem}\label{THM:main}
Let $ d=3 $.  Suppose that $\phi \in \HS(L^2; H^s)$ for some $s>\frac56$.  
Then, the defocusing stochastic cubic NLS \eqref{SNLS0} on $\R^3$
is globally well-posed in $H^s(\R^3)$.
\end{theorem}

Our main goal is to present an argument 
 which combines the $I$-method  in~\cite{CKSTT0}
with the Ito calculus approach in \cite{DD2}.
Note that  the regularity range  $s > \frac 56$ in Theorem \ref{THM:main}
 agrees with  the regularity range in the deterministic case~\cite{CKSTT0}.
We expect that this regularity range may be improved by employing more sophisticated tools such as the resonant decomposition \cite{CKSTT08};
see Remark \ref{REM:res}.
In view of the global well-posedness result in $H^1(\R^3)$ by de Bouard and Debussche~\cite{DD2}, 
we only consider $\frac 56 < s < 1$ in the following.

Let us first  go over the main idea of the $I$-method argument in \cite{CKSTT0}
applied to  the deterministic cubic NLS on $\R^3$, i.e.~\eqref{SNLS0} with $\phi = 0$.
Fix  $u_0 \in H^s(\R^3)$
for some $\frac 56 < s \leq 1$.
Then, the standard Strichartz theory yields
local well-posedness 
of \eqref{NLS1} with $u|_{ t= 0} = u_0$
in the subcritical sense,
namely, 
 time of local existence depends only on the $H^s$-norm of the initial data $u_0$.
Hence, once we obtain 
 an a priori control of the $H^s$-norm of the solution,
 we can  iterate the local-in-time argument and prove global existence.
 When $s = 1$, the conservation of the mass and energy in \eqref{cons}
 provides a global-in-time a priori control of the $H^1$-norm of the solution.
  When $\frac 56 < s < 1$,  the conservation of the energy $E(u)$ is no longer available
  (since $E(u) = \infty$ in general),
  while the mass $M(u)$ is still finite and conserved.
Therefore, the main goal is to control the growth of the homogeneous Sobolev $\dot H^s$-norm
of the solution.

Unlike the $s = 1$ case, 
we do not aim to obtain a global-in-time boundedness of the $\dot H^s$-norm of the solution. 
Instead, the goal is to show that, given any large target time $T\gg1$, 
the $\dot H^s$-norm of the solution remains finite 
on the time interval $[0, T]$, with a bound depending on $T$.
The main idea of the $I$-method is
to introduce a smoothing operator $I = I_N$, 
known as the $I$-operator, 
mapping $H^s(\R^3)$ into $H^1(\R^3)$.
Here, the $I$-operator depends on a parameter $N = N(T) \gg 1$
(to be chosen later)
such that $I_N $ acts essentially as the 
{\it i}\hspace{0.6pt}dentity operator on low frequencies $\{|\xi|\les N\}$
and as a fractional {\it i}\hspace{0.6pt}ntegration operator of order $1 - s$
on high frequencies $\{|\xi| \gg N\}$; see Section \ref{SEC:I} for the precise definition.
Thanks to the smoothing of the $I$-operator, 
the modified energy:
\begin{equation*}
E(I_N u)
=\frac{1}{2}\int_{\R^3}|\nb I_N u|^2 dx+\frac{1}{4}\int_{\R^3}|I_N u|^4dx
\end{equation*}

\noi
 is finite for $u \in H^s(\R^3)$.
Moreover, the modified energy $E(I_N u)$ controls  $\|u\|_{\dot H^s}^2$.
See \eqref{I2} below.
Hence, the main task is reduced to controlling the growth of the modified energy
$E(I_N u)$.

While the energy $E(u)$ is conserved for (smooth) solutions to NLS \eqref{NLS1}, 
the modified energy $E(I_N u)$ is no longer conserved since
$I_N u$ does not satisfy the original equation.
Instead, 
$I_N u$ satisfies
the following $I$-NLS:
\begin{align}\label{INLS}
\begin{split}
i\dt I_N u + \Dl I_N u 
& = I_N (|u|^2u) \\
& = |I_N u|^2 I_N u + \big\{I_N (|u|^2u) - |I_N u|^2I_N u\big\}\\
& =: \NN(I_N u)  + [I_N,  \NN](u),
\end{split}
\end{align}

\noi
where $\NN(u) = |u|^2 u$ denotes the cubic nonlinearity.
The commutator term
\begin{align}
[I_N,  \NN](u) = 
I_N (|u|^2u) - |I_N u|^2I_N u
\label{com1}
\end{align}

\noi
is the source of  non-conservation of the modified energy $E(I_N u)$.
A direct computation shows 
\[ \dt E(I_N u) =  - \Re \int_{\R^3} \cj {\dt I_N u }\, [I_N,  \NN](u) dx.\]

\noi
See \eqref{Igrowth}.
Thanks to the commutator structure, 
it is possible to obtain a good estimate (with a decay in the large parameter $N$)
for $\dt E(I_N u)$ on each local time interval
(See Proposition~4.1 in~\cite{CKSTT0}).
Then, by using a scaling argument (with a parameter $\ld = \ld(T) \gg1 $, depending on the target time $T$), 
we (i) first reduce the situation to the small data setting, 
 (ii) then iterate the local-in-time argument
with a good bound on $\dt E(I_N u^\ld)$ on the scaled solution $u^\ld$, 
and (iii) choose $N = N(T) \gg 1$ sufficiently large such that 
 the scaled target time $\ld^2 T$ 
is (at most) the doubling time
for the modified energy $E(I_N u^\ld)$.
This yields the regularity restriction $s > \frac 56$ in \cite{CKSTT0}.

Let us turn to the case of the stochastic NLS \eqref{SNLS0}.
In proceeding with the $I$-method, 
we need to estimate the growth of the modified energy $E(I_Nu)$.
In this stochastic setting, we have two sources for non-conservation of $E(I_N u)$.
The first one is 
  the commutator term
$[I_N, \NN](u)$ in \eqref{com1} as in the deterministic case described above.
This term can be handled almost in the same manner as in \cite{CKSTT0}
but some care must be taken due to a weaker regularity in time ($b < \frac 12$).
See Proposition~\ref{PROP:CommBd} below.
The second source for non-conservation of $E(I_N u)$ is  the stochastic forcing.
In particular, in estimating the growth 
of the modified energy $E(I_N u)$, we need to apply Ito's lemma 
to $E(I_N u)$, 
which introduces several correction terms.

In the deterministic case \cite{CKSTT0}, 
one iteratively applies the local-in-time argument and estimate energy increment
on each local time interval.
A naive adaptation of this argument to the stochastic setting
would lead to  iterative applications of Ito's lemma
to estimate the growth of the modified energy $E(I_N u)$.
In controlling an  expression of the form
\[\E \bigg[ \sup_{ 0 \leq t \leq t_0} E(I_N u)\bigg], \]

\noi
we need to apply 
Burkholder-Davis-Gundy inequality, 
which introduces a multiplicative constant $C>1$.
See Lemma \ref{LEM:Ito} below.
Namely, if we were to apply Ito's lemma iteratively 
on each time interval of local existence, 
then this would lead to an exponential growth 
of the constant in front of the modified energy.
This causes an iteration argument to break down.

We instead apply Ito's lemma only once
on the global time interval $[0, \ld^2T]$.
At the same time, we estimate the contribution
from the commutator term iteratively on each local time interval.
Note that this latter task requires 
a small data assumption, 
which we handle by introducing a suitable stopping time 
and iteratively verifying such a small data assumption.
See Section \ref{SEC:end}.

As in the deterministic setting, 
we employ a scaling argument to reduce the problem to the small data regime.
In the stochastic setting, we need to proceed with care in applying a scaling to the noise $\phi \xi$
since we need to apply Ito's lemma after scaling.
Namely, we need to  express the scaled noise 
as $\phi^\ld  \xi^\ld$, 
where $ \xi^\ld $ is another space-time white noise
(defined by the white noise scaling; see \eqref{noiseS} below)
such that Ito calculus can be applied.
This forces us to study the scaled Hilbert-Schmidt operator $\phi^\ld$.
In the application of Ito's lemma, 
there are correction terms due to 
$I_N \phi^\ld$
besides the commutator term
$[I_N, \NN](u^\ld)$.
In order to carry out an iterative procedure, 
we need to make sure that the contribution from 
the correction terms involving $I_N \phi^\ld$
is negligible as compared to that from the commutator term.
See Subsection \ref{SUBSEC:scal}
and Section \ref{SEC:end}.
As a result, the regularity restriction $s > \frac 56$
comes from the commutator term 
as in the deterministic case.

\smallskip

We conclude this introduction by several remarks.

\begin{remark}\rm

In this paper, we implement the $I$-method
in the stochastic PDE setting.
There is a recent work \cite{GKOT}
by Gubinelli, Koch, Tolomeo, and the third author, 
establishing  global well-posedness 
of the (renormalized) defocusing stochastic cubic nonlinear 
wave equation on the two-dimensional torus $\T^2$, 
forced by space-time white noise.
The $I$-method was also employed in~\cite{GKOT}.
We point out that our argument in this paper is a 
{\it genuine
extension of the $I$-method to the stochastic setting}, which can be applied
to a wide class of stochastic dispersive equations.
On the other hand, in \cite{GKOT}, 
the $I$-method was applied 
to the residual term $v = u - \Psi_{\text{wave}}$
in the Da Prato-Debussche trick \cite{DPD2}, 
where $\Psi_{\text{wave}}$ denotes 
the stochastic convolution in the wave setting.
Furthermore, the $I$-method argument in \cite{GKOT}
is pathwise, namely, {\it entirely deterministic} once
we take the pathwise regularity of $\Psi_{\text{wave}}$
(and its Wick powers)
from \cite{GKO}.

\end{remark}

\begin{remark}\rm
As in the usual application of the $I$-method in the deterministic setting, 
our implementation of the $I$-method in the stochastic setting yields
a polynomial-in-time  growth bound  of the $H^s$-norm of a solution.
See Remark \ref{REM:growth}.
We point out that the $I$-method approach to the singular
stochastic nonlinear wave equation on $\T^2$ with the defocusing cubic nonlinearity in \cite{GKOT} yields
a much worse double exponential growth bound.
\end{remark}

\begin{remark}\rm
In  a recent paper \cite{OO},  the third author
and Okamoto
studied  SNLS \eqref{SNLS} with an additive noise
in the mass-critical case ($p = 1+ \frac 4d$)
and the energy-critical case ($p = 1+ \frac 4{d-2}$, $d \geq 3$).
By adapting the recent deterministic mass-critical
and energy-critical global theory, 
they proved global well-posedness of 
\eqref{SNLS} in the critical spaces.
In particular, when $d = 2$ and $p = 3$, 
this yields global well-posedness
 the two-dimensional defocusing stochastic cubic NLS  in  $L^2(\R^2)$.
This is the reason why 
we only considered the three-dimensional case
in Theorem \ref{THM:main},
since our $I$-method argument would  yield global well-posedness only for $s > \frac 47$ 
in the two-dimensional cubic case
(just as in the deterministic case \cite{CKSTT0}),
which is subsumed by the aforementioned
global well-posedness result in \cite{OO}.

\end{remark}

\begin{remark} \label{REM:res}
\rm
In an application of the $I$-method, 
it is possible to introduce a correction term (away from a nearly resonant part)
and improve the regularity range.  See  \cite{CKSTT08}.
It would be of interest to implement
such an argument 
to the stochastic PDE setting since
a computation of a correction term would involve Ito's lemma.

\end{remark}

\begin{remark}\rm
We mentioned that our implementation of the $I$-method
in the stochastic PDE setting is 
sufficiently general and is applicable to other  dispersive equations
forced by additive noise.
This is conditional to an assumption 
that a commutator term can be treated with  a weaker temporal regularity $b < \frac 12$.
In the case of SNLS, this can be achieved by a simple interpolation argument, 
at a slight loss of spatial regularity.
See Section~\ref{SEC:com}.
See also \cite{CM}
for an analogous argument in the periodic case.
In this regard, it is of interest to study the stochastic KdV equation
in negative Sobolev spaces since
 crucial estimates for KdV
require 
the temporal regularity 
to be  $b = \frac 12$.
  See \cite{BO93x, KPV, CKSTT3}.

\end{remark}

This paper is organized as follows.
In Section \ref{SEC:Prelim}, 
we go over the preliminary materials 
from deterministic and stochastic analysis.
We then reduce a proof of Theorem \ref{THM:main}
to controlling the homogeneous $\dot H^s$-norm
of a solution (Remark \ref{REM:L2bound}).
In Section \ref{SEC:I}, we introduce the $I$-operator
and go over local well-posedness of $I$-SNLS \eqref{SNLSI}.
Then, we discuss the scaling properties
of $I$-SNLS in Subsection \ref{SUBSEC:scal}.
In Section \ref{SEC:com}, 
we briefly go over the nonlinear estimates, 
indicating required modifications 
from \cite{CKSTT0}.
In Section \ref{SEC:energy}, 
we apply Ito calculus to bound the modified energy
in terms of a term involving the commutator $[I_N, \NN]$.
Lastly, 
we put all the ingredients together and present a proof of Theorem \ref{THM:main}
in Section \ref{SEC:end}.

\section{Preliminaries}\label{SEC:Prelim}

In this section, we first introduce notations
and function spaces along with the relevant linear estimates.
We also go over preliminary lemmas from stochastic analysis.
We then discuss a reduction of the proof of Theorem \ref{THM:main};
see Remark \ref{REM:L2bound}.

\subsection{Notations}

 For simplicity, we  drop $ 2\pi $ in dealing with the Fourier transforms.
We first
recall the Fourier restriction norm spaces
 $ X^{s,b}(\R\times \R^d) $ introduced by Bourgain~\cite{Bo93}. 
 The $X^{s, b}$-space is defined by the norm:
\begin{align}
\|u\|_{X^{s,b}}=\| \jb{\xi}^{s}\jb{\tau+|\xi|^2}^b \ft{u}(\tau,\xi)\|_{L^2_{\tau}L^2_{\xi}(\R\times\R^d)}, 
\label{Xsb}
\end{align}

\noi
where $ \jb{\,\cdot\,} = (1 +|\cdot|^2)^\frac{1}{2} $. 
When $b > \frac 12$, we have the following embedding:
\begin{align}
X^{s, b}(\R \times \R^d ) \subset C(\R; H^s(\R^d)).
\label{embed1}
\end{align}

\noi
Given $\dl > 0$, we define the local-in-time version $ X^{s,b}_{\dl} $ on $ [0,\dl] \times\R^d  $ by 
\begin{align}
\|u\|_{X^{s,b}_{\dl}}:=\inf \big\{ \|v\|_{X^{s,b}(\R \times \R^d)} : v|_{[0,\dl]}=u    \big\}.
\label{embed2}
\end{align}

\noi
Given a time  interval $J \subset \R$, 
we also define  the local-in-time version $  X^{s,b}(J)$
in an analogous manner.

When we work with space-time function spaces, we use short-hand notations such as
 $C_T H^s_x  = C([0, T]; H^s(\R^d))$.

%
%
%
%

We write $ A \les B $ to denote an estimate of the form $ A \leq CB $. 
Similarly, we write  $ A \sim B $ to denote $ A \les B $ and $ B \les A $ and use $ A \ll B $ 
when we have $A \leq c B$ for small $c > 0$.
We may use subscripts to denote dependence on external parameters; for example,
 $A\les_{p, q} B$ means $A\le C(p, q) B$,
 where the constant $C(p, q)$ depends on parameters $p$ and $q$. 
We also use  $ a+ $ (and $ a- $) to mean  $ a + \eps $ (and $ a-\eps $, respectively)
 for arbitrarily small $ \eps >0 $.
As it is common in probability theory, we use $A\wedge B$ to denote $\min (A, B)$.

Given dyadic $M \geq 1$, 
we use $\P_M$ to denote 
  the Littlewood-Paley projector onto the frequencies\footnote{When $M = 1$, 
  $\P_1$ is a smooth projector onto the frequencies $\{|\xi|\les 1\}$.} $\{|\xi|\sim M\}$
  such that 
\begin{align}
 f = \sum_{\substack{M\geq 1\\\text{dyadic}}}^\infty \P_M f. 
\label{Px1}
 \end{align}

In view of the time reversibility  of the problem, 
we only consider positive times in the following.

\subsection{Linear estimates}

We first  recall the Strichartz estimate. We say that
a pair of indices  $(q, r)$ is Strichartz admissible
if $2\leq q, r \leq \infty$, $(q, r, d) \ne (2, \infty, 2)$, and 
\begin{equation*}
\frac{2}{q} + \frac{d}{r} = \frac{d}{2}.
\label{Admin}
\end{equation*}
Then, given any  admissible pair $(q,r)$, the following Strichartz estimates
are known to hold: 
\begin{align}\label{Str1}
\|S(t)f\|_{L_t^qL_x^r(\R\times \R^d)}&\lesssim \|f\|_{H^s}.
\end{align}

\noi
See \cite{Strichartz, Yajima, GV, KeelTao}.

Next, we recall the standard linear estimates for  the $X^{s, b}$-spaces.
See, for example, \cite{GTV, Tao} for the proofs of (i) and (ii).

\begin{lemma}\label{LEM:lin}
%
%
%
\textup{(i) (homogeneous linear estimate).}
Given $s, b \in \R$, we have 
\begin{align*}  
\|S(t)f\|_{X^{s,b}_{T}}\les\|f\|_{H^s}
\end{align*}

\noi
for any $0 < T \leq 1$.
Moreover, we have $S(t) f \in C([0, T]; H^s(\R^d))$ for $f \in H^s(\R^d)$.

\smallskip

\noi
\textup{(ii) (nonhomogeneous linear estimate).} 
Given $s \in \R$, $b > \frac 12$
sufficiently close to $\frac 12$, 
 and small $\ta > 0$, 
we have
\begin{align*}
\bigg\|\int^{t}_{0}S(t-t')F(t')dt'\bigg\|_{X^{s, b }_T}\les T^{\ta}\|F\|_{X^{s, b - 1 + \ta}_T}
\end{align*}

\noi
for any $0 < T \leq 1$.

\smallskip

\noi
\textup{(iii)} \textup{(transference principle).} Let $(q, r)$ be Strichartz admissible.
Then, for any $b > \frac 12 $, we have
\[ \| u \|_{L^q_tL^r_x} \les \| u \|_{X^{0, b}}.\]

\smallskip

\noi
\textup{(iv)} Let $ d=3 $.
Then,  given any $2\leq p <\frac{10}{3}$,  
there exists small $\eps > 0$ such that 
\begin{align}\label{Str2}
\|u\|_{L_{t,x}^{p}(\R\times \R^3)} \lesssim \|u\|_{X^{0,\frac{1}{2}-\eps}}.
\end{align}
\end{lemma}

\begin{proof}
As for (iii), see, for example, Lemma 2.9 in \cite{Tao}.
In the following, we only discuss Part (iv).
Noting that $(\frac {10}3, \frac {10}3)$ is Strichartz admissible
when $ d= 3$, it follows from 
the transference principle 
and the Strichartz estimate \eqref{Str1} that 
\begin{align}
\|u\|_{L_{t, x}^{\frac{10}{3}}}&\lesssim \|u\|_{X^{0,b}}
\label{Str3}
\end{align}

\noi
for  $ b >  \frac{1}{2} $.
Interpolating this with 
 the trivial bound: $\|u\|_{L_{t,x}^{2}} =   \|u\|_{X^{0,0}}$, 
 we obtain the desired estimate \eqref{Str2}.
\end{proof}

\subsection{Tools from stochastic analysis}
\label{SUBSEC:Stco}

Lastly, we go over basic tools from stochastic analysis
and then provide some reduction for the proof of Theorem \ref{THM:main}.

We first recall the regularity properties of the stochastic convolution $\Psi$ defined in \eqref{SNLS3}.
Given two separable Hilbert spaces $H$ and $K$, we denote by $\HS (H;K)$ the space of Hilbert-Schmidt operators $\phi$ from $H$ to $K$, endowed with the norm:
\[
\| \phi \|_{\HS(H;K)} = \bigg( \sum_{n \in \N} \| \phi f_n \|_K^2 \bigg)^{\frac{1}{2}},
\]

\noi
where $\{ f_n \}_{n \in \N}$ is an orthonormal basis of $H$.
Recall that the Hilbert-Schmidt norm of $\phi$ is independent of
the choice of an orthonormal basis of $H$.

Next, recall the definition of a cylindrical Wiener process $ W $ on $ L^2(\R^d) $.
 Let $(\O, \F, P)$ be  a probability space  endowed with a filtration $\{ \F_t \}_{t \ge 0}$.
Fix an orthonormal basis $\{ e_n \}_{n \in \N}$ of $L^2(\R^d)$.
We define an $L^2 (\R^d)$-cylindrical Wiener process $W$ by 
\begin{align}
W(t) = \sum_{n \in \N} \beta_n (t)  e_n ,
\label{Psi1}
\end{align}

\noi
where $\{ \beta_n \}_{n \in \N}$ is a family of mutually independent complex-valued Brownian motions\footnote{Namely, 
the real and imaginary parts of $\be_n$ are  independent (real-valued) Brownian motions.} associated with the filtration $\{ \F_t \}_{t \ge 0}$ .
Note that a space-time white noise $\xi$ is given by a distributional derivative (in time) 
of $W$.
Hence, 
we can express the stochastic convolution~$\Psi$ in \eqref{SNLS3} as
\begin{align}
 \begin{split}
\Psi (t) & =  - i \int_0^t S(t - t') \phi dW (t')\\
& = -i \sum_{n \in \N} \int_0^t S(t-t') \phi e_n  \, d \beta_n (t').
\end{split}
\label{Psi2}
\end{align}

%

The next lemma summarizes
the regularity properties
of the stochastic convolution.

\begin{lemma}\label{LEM:stoconv}
Let $ d \geq 1 $, $ T > 0 $, and $ s \in \R $. Suppose that $ \phi \in \HS(L^2;H^s) $.

\smallskip

\noi

\begin{itemize}
\item[\textup{(i)}]
We have $ \Psi \in C([0,T];H^s(\R^d)) $ almost surely.

\smallskip

\item[\textup{(ii)}]
Given any  $1 \leq q < \infty$ and finite $r \geq 2$ such that $ r \le \frac{2d}{d-2}$ when $d \geq 3$, 
we have
$\Psi \in L^q([0, T];  W^{s, r}(\R^d))$
almost surely.

\smallskip

\item[\textup{(iii)}]
Given $ b < \frac 12$, 
we have $\Psi\in X^{s,b}([0,T])$ almost surely. Moreover,  there exists $\ta > 0$ such that
\begin{align}
\E\Big[  \|\Psi\|^{p}_{X^{s,b}([0,T])}    \Big] 
\les p^\frac{p}{2} \jb{T}^{\ta p}  \|\phi\|^{p}_{\HS(L^2;H^s)}
\label{Xsb2}
\end{align}

\noi
for any finite $p \geq 1$.

\end{itemize}

\end{lemma}

Regarding the proof of Lemma \ref{LEM:stoconv}, 
see~\cite{DZ} for~(i)
and~\cite{DD2, OPW} for (ii).
As for (iii), 
see \cite[Proposition 2.1]{DDT}, 
\cite[Proposition 4.1]{O09},  and \cite[Lemma 3.3]{CM} for the proofs of the $X^{s, b}$-regularity of the stochastic
convolution.
The works  \cite{DDT, O09, CM}  treat  a different equation (the KdV equation)
and/or a different setting (on the circle) but the proofs can be easily adapted to our context.
For example, one can follow the argument in the proof of 
Proposition~2.1 in~\cite{DDT}
to obtain \eqref{Xsb2} for $p = 2$.
Then, by noting from~\eqref{Psi2}
that the stochastic convolution~$\Psi$
is nothing but (a limit of) a linear combination of Wiener integrals, 
the general case follows from the $p =2 $ case and the Wiener chaos estimate;\footnote{The proof of Proposition 2.1 in \cite{DDT} reduces to bounding the second moment of the $H^b$-norm
of a certain (scalar) Wiener integral.
Similarly, the proof of Lemma \ref{LEM:stoconv}\,(iii) for general finite $p \geq 1$ reduces
to bounding the $p$th moment of the $H^b$-norm of the same Wiener integral. 
By the Wiener chaos estimate, 
this can be further reduced to bounding the second moment.}
see, for example,  \cite[Lemma 2.5]{GKO2}.  See also 
\cite[Theorem I.22]{Simon} and \cite[Proposition 2.4]{TTz}.

Once we have Lemma \ref{LEM:stoconv}, 
we can use the Strichartz estimates \eqref{Str1} (without the $X^{s, b}$-spaces)
to  prove
 local well-posedness of SNLS 
 \eqref{SNLS0}
in $H^s(\R^3)$ for $s \geq s_\text{crit} = \frac 12$, 
provided that $\phi \in HS(L^2; H^s)$.
See \cite{DD2, OO}.
In particular, for the subcritical range $s > \frac 12$, 
the random time $\dl = \dl(\o) $ of local existence, starting from $t = t_0$,  
satisfies
\begin{align*}
\dl \ges \Big( \| u(t_0)\|_{H^s} + C_{t_0}(\Psi) \Big)^{-\ta}
\end{align*}

\noi
for some $\ta > 0$, 
where 
$C_{t_0}(\Psi)>0$ denotes certain Strichartz norms
of the stochastic convolution $\Psi$, restricted to a time interval $[t_0, t_0 + 1]$.
Given $T > 0$, it follows from Lemma~\ref{LEM:stoconv}
that $C_{t_0}(\Psi)$ remains finite almost surely for any $t_0 \in [0, T]$. 
Therefore, 
Theorem~\ref{THM:main} follows once we show that 
$\sup_{t \in [0, T]} \| u(t) \|_{H^s}$ remains finite almost surely for any $T>0$
(with a bound depending on $T>0$).

Lastly, we recall the a priori mass control from \cite{DD2}
whose proof follows from Ito's lemma applied to the mass $M(u)$
in \eqref{cons}
and Burkholder-Davis-Gundy inequality
(see \cite[Theorem 4.36]{DZ}).

\begin{lemma} \label{LEM:bound}
Assume $ \phi \in \HSi $
 and  $ u_0\in L^2(\R^3) $. 
Let  $u$ be the solution to  SNLS~\eqref{SNLS0} with $u|_{t = 0}=u_0$
and 
 $T^{\ast} = T^{\ast}_{\omega} (u_0)$ be the forward maximal time of existence.
Then, given $T>0$, 
there exists $C_1 = C_1 (M(u_0), T, \| \phi \|_{\HS (L^2; L^2)})>0$ 
such that for any stopping time $\tau$ with $0<\tau< \min (T^{\ast}, T)$ almost surely, we have
\begin{align}
\E \bigg[ \sup_{0\le t \le \tau} M(u(t)) \bigg] \le C_1.
\label{M1}
\end{align}

\end{lemma}

\begin{remark}\label{REM:L2bound}
\rm 
In view of Lemma \ref{LEM:bound}, 
given finite $T > 0$, the $L^2$-norm of the solution remains
bounded almost surely on  $[0, T^*_\o\wedge T]$, 
where $T^*_\o$ is the forward maximal  time 
of existence.\footnote{\label{FT1}It can be shown  that the left-hand side
of \eqref{M1} grows at most linearly in $T$.
See, for example, the proof of Proposition 3.2 in \cite{DD2}
and the proof of Proposition 5.1 in \cite{CM}.}
Therefore, it follows from the discussion above that, 
in order to prove Theorem~\ref{THM:main},
it suffices to  show that 
the homogeneous Sobolev norm
$ \| u(t) \|_{\dot H^s}$ remains finite 
almost surely
on each bounded time interval $[0, T]$.
In the following, our analysis involves only homogeneous Sobolev spaces.

\end{remark}

\section{$I$-operator, $ I $-SNLS, and their scaling properties}
\label{SEC:I}

\subsection{$I$-operator}\label{SUBSEC:Iop}

In \cite{Bo98}, Bourgain introduced the so-called high-low method
in establishing global well-posedness of the defocusing cubic NLS on $\R^2$
below the energy space.
The high-low method is based on truncating the dynamics by a sharp frequency cutoff
and separately studying the low-frequency and high-frequency dynamics.
In \cite{CKSTT0}, 
Colliander, Keel, Staffilani, Takaoka,  and Tao
proposed to use a smooth positive frequency multiplier instead.

Let $0 < s < 1$.
Given $N \geq 1$, we define a 
smooth, radially symmetric, non-increasing (in $|\xi|$)
multiplier $m_N$, satisfying
\begin{equation}
m_N(\xi)=
\left\{\begin{array}{ll}
1, & \text{for } |\xi|\le N, \\
\big(\frac{N}{|\xi|}\big)^{1-s}, & \text{for } |\xi|\ge2N.
\end{array}\right.
\label{Ix1}
\end{equation} 

\noi
Since $m_N$ is radial, 
with a slight abuse of notation, 
we may use the notation $m_N(|\xi|)$
by viewing $m_N$ as a function on $[0, \infty)$.

We then define the $I$-operator $I = I_N$
to be the Fourier multiplier operator with the multiplier $m_N$:
\begin{align}
\ft{I_Nf}(\xi)=m_N(\xi)\ft{f}(\xi).
\label{Ix2}
\end{align}

\noi
As mentioned in the introduction, 
$I_N$ acts as the identity operator on low frequencies $\{ |\xi| \leq N\}$, 
while it acts as a fractional  
integration operator of order $1-s$ on high frequencies $\{ |\xi| \geq 2N\}$.
As a result, we have the following bound:
\begin{align}
\|f\|_{\dot H^s}\les \|f\|_{L^2} + \|I_N f\|_{\dot H^1}
\qquad \text{and} \qquad 
\|I_N f\|_{\dot H^1}\les  N^{1-s}\|f\|_{\dot H^s}.
\label{I2}
\end{align}

\subsection{$I$-SNLS}
\label{SUBSEC:lSNLSI}

By applying the $I$-operator to SNLS \eqref{SNLS0}, 
we obtain the following $I$-SNLS:
\begin{equation}\label{SNLSI}
\begin{cases}
i\dt I_Nu + \Dl I_Nu=I_N(|u|^2u)+I_N\phi \xi    \\ 
I_Nu|_{t=0}=I_Nu_0\in H^1(\R^3).
\end{cases}
\end{equation}

\noi
In this subsection, we study  local well-posedness  of 
the Cauchy problem~\eqref{SNLSI}. 
A similar local well-posedness result for the (deterministic) $I$-NLS
(namely, \eqref{SNLSI} with $\phi = 0$)
was  studied in \cite[Proposition 4.2]{CKSTT0}. 
In order to capture the temporal regularity 
of the stochastic convolution
(Lemma \ref{LEM:stoconv}), 
we need to work with the $X^{s, b}$-space with $b < \frac 12$
and hence need to establish a trilinear estimate in this setting.
See Lemma \ref{LEM:Trilin} below.
The following proposition  allows us to avoid using the
$ L^2$-norm which is supercritical with respect to scaling
(as in \cite{CKSTT0}).

\begin{proposition}\label{PROP:LWP}
Let $\frac{1}{2}<s<1$, $ \phi\in \HS(L^2;\dot{H}^s) $, and
$u_0 \in \dot H^s(\R^3)$. 
Then, there exist an almost surely positive stopping time 
\[ \dl=\dl_\o\big(\|I_N u_0\|_{\dot H^1}, \|I_N \phi\|_{\HS(L^2; \dot H^1)}\big) \]

\noi
 and a unique local-in-time solution $ I_N u\in C( [0,\dl];\dot{H}^1(\R^3) ) $ to  $ I $-SNLS \eqref{SNLSI}. 
Furthermore, if $ T^*=T^*_{\o} $ denotes the forward maximal time of existence, 
 the following blowup alternative holds:
\begin{align}
T^*=\infty \qquad \text{or} \qquad \lim_{T\nearrow T^*} \|I_N u\|_{L^\infty_{T}\dot{H}^1_x}=\infty.
\label{LWP2}
\end{align}

\end{proposition}

Proposition \ref{PROP:LWP} follows from a standard contraction argument once we prove
the following trilinear estimate.

\begin{lemma}\label{LEM:Trilin}
Let $\frac{1}{2}<s<1$.   Then, there exists small $\eps > 0$ such that 
\begin{align}
\|\nb I_N(u_1\cj{u_2}u_3)\|_{X^{0,-\frac{1}{2}+2\eps}_T}
\les \prod_{j=1}^{3}\|\nb I_N u_j\|_{X^{0,\frac{1}{2}-\eps}_T}
\label{tri1}
\end{align}

\noi
for any $0 \leq T \leq 1$, 
where the implicit constant is independent of  $ N \geq 1$.
\end{lemma}

As compared to Proposition 4.2 in \cite{CKSTT0}, 
we need to work with a slightly weaker  temporal regularity on the right-hand side of \eqref{tri1}.

Before going over a proof of Lemma \ref{LEM:Trilin}, 
let us briefly discuss a proof of Proposition \ref{PROP:LWP}.
By writing \eqref{SNLSI} in the Duhamel formulation, we have
\begin{align*}
\begin{split}
I_N u(t) 
& = \Phi(I_Nu) \\
: \!& = S(t) I_N u_0 
- i \int_0^t S(t - t') I_N(|u|^2 u)(t') dt' + I_N \Psi(t), 
\end{split}
\end{align*}

\noi
where
$\Phi =  \Phi_{I_N u_0, I_N \phi}$
and  we interpreted the nonlinearity as
a function of $I_Nu$:
\[ I_N(|u|^2 u)
= I_N(|I_N^{-1} (I_N u)|^2 I_N^{-1} (I_N u)).\]

\noi
Fix small $\eps > 0$.
Then, by Lemmas \ref{LEM:lin} and \ref{LEM:stoconv}
followed by Lemma \ref{LEM:Trilin}, 
we have
\begin{align}
\begin{split}
\|\nb \Phi(I_Nu) \|_{X^{0, \frac{1}{2}-\eps}_\dl}
& \le \| \nb S(t) I_N u_0 \|_{X^{0, \frac{1}{2}-\eps}_\dl}\\
& \hphantom{X}
+ \bigg \| \nb \int_0^t S(t - t') I_N(|u|^2 u)(t') dt' \bigg\|_{X^{0, \frac{1}{2}+\eps}_\dl}
+  \| \nb I_N \Psi\|_{X^{0, \frac{1}{2}-\eps}_\dl}\\
& \les \| I_N u_0 \|_{\dot H^1}
+ \dl^\eps \| \nb I_N(|u|^2 u)\|_{X^{0, - \frac12 + 2 \eps}_\dl}  + C_\o \| I_N \phi \|_{\HS(L^2; \dot H^1)}\\
& \les \| I_N u_0 \|_{\dot H^1}
+ C_\o \| I_N \phi \|_{\HS(L^2; \dot H^1)}
+ \dl^\eps \| \nb I_N u\|_{X^{0, \frac 12 - \eps}_\dl}^3  
\end{split}
\label{ZX2}
\end{align}

\noi
for an almost surely finite random constant $C_\o>0$
and for any $0 \leq \dl \leq 1$.
Similarly, we have 
\begin{align}
\begin{split}
\|\nb (\Phi( & I_Nu)  - \Phi(I_Nv)) \|_{X^{0, \frac 12 - \eps}_\dl}\\
& \les \dl^\eps \Big(\| \nb I_N u\|_{X^{0, \frac 12 - \eps}_\dl}^2  +\| \nb I_N v\|_{X^{0, \frac 12 - \eps}_\dl}^2  \Big)
\| \nb (I_N u - I_N v)\|_{X^{0, \frac 12 - \eps}_\dl}. 
\end{split}
\label{ZX3}
\end{align}

\noi
From \eqref{ZX2} and \eqref{ZX3}, we conclude that 
$\Phi$ is almost surely a contraction
on the ball of radius 
\[R = 2\Big( \| I_N u_0 \|_{\dot H^1}
+ C_\o \| I_N \phi \|_{\HS(L^2; \dot H^1)}\Big)\]

\noi
 in $\nb^{-1} X^{0, \frac 12 - \eps}$
  by choosing $\dl = \dl_\o(R) >0$ sufficiently small.
Moreover,  from 
\eqref{embed1} and
Lemmas~\ref{LEM:lin} and~\ref{LEM:stoconv}
with~\eqref{ZX2}, we also conclude that 
$ I_N u \in C( [0,\dl];\dot{H}^1(\R^3) )$.
 This proves Proposition~\ref{PROP:LWP}.
The following remark plays an important role 
in iteratively applying the local-in-time argument in 
Section \ref{SEC:end}.

\begin{remark}\rm\label{REM:local}

The argument above shows that 
there exist small $\eta_0, \eta_1 > 0$ such that 
if, for 
a given interval $ J = [t_0, t_0 + 1] \subset [0, \infty)$ of length $ 1 $
and  $ \o\in \O $, we have 
\begin{align}
E(I_N u(t_0)) \leq \eta_0
\qquad \text{and} \qquad
\|\nb I_N \Psi(\o)\|_{X^{0,\frac{1}{2}-\eps}(J)}\leq \eta_1, 
\label{ZX4}
\end{align} 

\noi
then a solution $I_N u $ to $I$-SNLS \eqref{SNLSI}
exists on the interval $J$ with the bound:
\begin{align*}
\|\nb I_N u\|_{X^{0,\frac{1}{2}-\eps}(J)}\le C_0
\end{align*}

\noi
for some absolute constant $C_0$,  
uniformly in $ N \geq 1. $
\end{remark}

We now present a proof of Lemma \ref{LEM:Trilin}.

\begin{proof}[Proof of Lemma \ref{LEM:Trilin}]
By the interpolation lemma (\cite[Lemma 12.1]{CKSTT2}), 
it suffices to prove~\eqref{tri1} for $N = 1$.
Let $I = I_1$.
By the definition \eqref{embed2} 
of the time restriction norm, duality,  and Leibniz rule for $ \nb I $, it suffices to show that\footnote{Here, 
we are essentially using the triangle inequality
$ \langle\xi_1+\cdots+ \xi_4\rangle^s\lesssim \langle\xi_1 \rangle^s +\cdots \langle\xi_4 \rangle^s $
for $s\geq 0$ and  the fact that $X^{s, b}$ is a Fourier lattice.}
\begin{align}\label{Trilin3}
\bigg|\iint_{\R\times\R^3}\big(\nb I u_1\big)\cj{u_2}u_3u_4 \,dx dt\bigg|
\les \prod_{j=1}^{3}\|\nb Iu_j\|_{X^{0,\frac{1}{2}-\eps}} \|u_4\|_{X^{0,\frac{1}{2}-2\eps}}.
\end{align}
For $j\in\{2,3\}$, we split the functions $u_j$ into high and low frequency components:
\begin{align}
u_j=u_j^\hi+u_j^{\lo},
\label{HL}
\end{align}

\noi
where the spatial Fourier supports
of $u_j^\hi$ and $u_j^\lo$ are contained in $\{|\xi|\geq \frac 12  \}$ and  $\{|\xi|\leq1 \}$, respectively.

By noting $ u_j^\lo=Iu_j^\lo$
and Sobolev's inequality
(both in space and time), we have
\begin{align}
\|u_j^\lo \|_{L^6_{t,x}}\les \|\nb I u_j\|_{X^{0,\frac{1}{2}-}}.
\label{Trilin3a}
\end{align} 

As for $u_j^\hi$, we claim 
\begin{align}
\|{u}_j^\hi\|_{L^{5+}_{t,x}}\les \|\nb I u_j\|_{X^{0,\frac{1}{2}-}}.
\label{Trilin4}
\end{align}

\noi
Since $N = 1$, we have  $I\sim |\nb |^{s-1}$.
Then, by Sobolev's inequality 
and the transference principle
(Lemma \ref{LEM:lin}\,(iii)) with 
an admissible pair $(q, r ) =\big(5+, \frac{30}{11}-\big)$, 
we have 
\begin{align}\label{Trilin5}
\begin{split}
\|u_j^\hi \|_{L^{5+}_{t,x}} 
& = 
\big\| |\nb|^{1-s} I u_j^\hi\big\|_{L^{5+}_{t,x}}
\les \big\| \jb{\nb}^{s-}  |\nb|^{1-s}  Iu_j^\hi  \big\|_{L^{5+}_{t}L^{\frac{30}{11}-}_x}\\
& \les \big\| |\nb|^{1-} Iu_j^{\hi}\big\|_{X^{0,\frac{1}{2}+}}, 
\end{split}
\end{align}

\noi
provided that 
$s > \frac 12$.
On the other hand, by Sobolev's inequality, we have 
\begin{align}
 \big\| | \nb  |^{1-s} Iu_j^\hi\big\|_{L^{5+}_{t, x}}
 \les \big\|  |\nb|^{\frac{19}{10} - s+}  Iu_j^\hi  \big\|_{X^{0,\frac{3}{10}+}}.
\label{Trilin6}
\end{align}

\noi
By 
interpolating \eqref{Trilin5} and \eqref{Trilin6}, 
we obtain \eqref{Trilin4}.

We now estimate
 \eqref{Trilin3}
by expanding $u_j$, $j = 2, 3$, 
as $u_j^\hi+u_j^{\lo}$.
For  $j = 2, 3$, 
let
$p_j = 6$ in treating $u_j^\lo$
and 
$p_j = 5+$ in treating $u_j^\hi$.
Then, 
the claimed estimate \eqref{Trilin3} follows
from 
$ L^{\frac{10}{3}-}_{t,x},  L^{p_2}_{t,x},  L^{p_3}_{t,x},  L^{p_4}_{t,x}$-H\"older's inequality, 
Lemma \ref{LEM:lin}\,(iv),  
\eqref{Trilin3a},  and \eqref{Trilin4},
where $p_4$ is defined by $\frac{1}{p_4} = 1 - \big(\frac{3}{10-}\big) - \frac{1}{p_2}
- \frac{1}{p_3}$ such that $ 2 \leq p_4 <  \frac {10}3$.
\end{proof}

\subsection{Scaling property}\label{SUBSEC:scal}

In this subsection, we discuss the scaling properties of SNLS~\eqref{SNLS0}
and  $ I $-SNLS~\eqref{SNLSI}.
Before doing so, we first recall  the scaling property of the (deterministic) cubic NLS: 
\begin{align}
i\dt u +\Dl u =|u|^2u .
\label{NLS}
\end{align}

\noi
This equation enjoys the following scaling invariance;
if $u$ is a solution to \eqref{NLS}, 
then the scaled function 
\begin{align}\label{scaling}
u^\ld (t,x) := \ld^{-1}  u(\ld^{-2}t ,\ld^{-1} x)
\end{align}

\noi
also satisfies the equation \eqref{NLS} with the scaled initial data.
In the application of the $I$-method in the deterministic case
(as in \cite{CKSTT0}), we apply this scaling first 
and then apply the $I$-operator
to obtain $I$-NLS \eqref{INLS} (with $u^\ld$ in place of $u$).

In our current stochastic setting, 
when we apply the scaling, we also need to scale the noise $\phi \xi$.
In order to apply Ito calculus
to the scaled noise, 
we need to make sure that the scaled noise is  given by 
another space-time white noise $\xi^\ld$
(with a scaled Hilbert-Schmidt operator $\phi^\ld$).
  For this purpose, we first recall the scaling property
of a space-time white noise.
Given a space-time white noise $\xi$
on $\R\times \R^d$, 
it is well known that 
the scaled noise $\xi^\ld$ defined by\footnote{Since $\xi$ is merely
a distribution, a pointwise evaluation does not quite make sense.
Strictly speaking, we need to apply
the (inverse) scaling to test functions.
For simplicity, however, we use this slightly abusive notation.} 
\begin{align}\label{noiseS}
\xi^\ld(t,x)=\xi_{a_1,a_2}^\ld (t, x) := \ld^{-\frac{a_1+da_2}{2}}
\xi(\ld^{-a_1}t,\ld^{-a_2}x)
\end{align}

\noi
is also a space-time white noise
for any $a_1, a_2 \in \R$.

Next, let us study the scaling property of the Hilbert-Schmidt operator
$\phi$ via its kernel representation.
Recall from 
\cite[Theorem VI.23]{RS}
that a bounded linear operator $\phi$ on $L^2(\R^3)$ is 
Hilbert-Schmidt if and only if it is represented
as an integral operator with 
a  kernel $k \in L^2(\R^3 \times \R^3)$:
\begin{align*}
(\phi f)(x) = \int_{\R^3} k(x, y) f(y) dy
\end{align*}

\noi
with
$ \|\phi\|_{\HS(L^2; L^2)} 
 = \| k \|_{L^2_{x, y}}.$
More generally, we have 
\begin{align}
 \|\phi\|_{\HS(L^2; \dot{H}^s)} 
 = \| k \|_{\dot{H}^s_xL^2_y}.
 \label{HS2a}
\end{align}

\noi
With this in mind,  let us evaluate 
 $\phi\xi$ 
at $(\frac{t}{\ld^2}, \frac{x}{\ld})$ with a factor of $\ld^{-3}$.
By a change of variables
and \eqref{noiseS} with $(a_1, a_2) = (2, 0)$, we have 
\begin{align}
\begin{split}
\ld^{-3}\phi \xi\big(\tfrac{t}{\ld^2}, \tfrac{x}{\ld}\big) 
 &= \ld^{-3 } \int_{\R^3} k\big(\tfrac{x}{\ld}, y  \big)
  \xi \big(\tfrac{t}{\ld^2}, y \big)  dy\\
 &= \ld^{-2 } \int_{\R^3} k\big(\tfrac{x}{\ld}, y  \big) 
\xi^\ld (t, y) dy .
\end{split}
 \label{HS3}
\end{align}

\noi
This motivates us to define the scaled kernel $k^\ld$ by 
\begin{align}
 k^{\ld}(x, y )  = 
 \ld^{-2 }
k(\ld^{-1}x, y )
\label{HS4}
\end{align}

\noi
and the associated 
Hilbert-Schmidt operator $\phi^\ld$
with an integral kernel $k^\ld$.
Then, it follows from \eqref{HS3} and \eqref{HS4} that 
\begin{align}
\ld^{-3}\phi \xi\big(\tfrac{t}{\ld^2}, \tfrac{x}{\ld}\big) 
= \phi^\ld \xi^\ld(t, x).
 \label{HS5}
\end{align}

\noi
Therefore, by applying the scaling \eqref{scaling} with \eqref{HS5}
to SNLS \eqref{SNLS0}
and then applying the $I$-operator,  we obtain
\begin{align}
i \dt I_N u^\ld  + \Dl I_N u^\ld =  I_N (|u^\ld|^2 u^\ld) +  I_N \phi^\ld \xi^\ld.
\label{SNLSI2}
\end{align}

\noi
In the following lemma,  we record the scaling property of the 
Hilbert-Schmidt norm of $I_N \phi^\ld$.

\begin{lemma}\label{LEM:scaling}
Let $d = 3$, $ 0<s<1 $, and  $ \phi \in \HS(L^2,\dot{H}^s)$.
 Then, we have
\begin{align}\label{sc-hs}
\|I_N \phi^\ld\|_{\HS(L^2;\dot{H}^1)}\les N^{1-s}\ld^{-\frac 12 - s} \|\phi\|_{\HS(L^2;\dot{H}^s)}.
\end{align}

\noi
As a consequence, 
given any $\eps > 0$, there exists $\ta > 0$  such that
\begin{align}\label{scale-sto-co}
\Big\| \| \nb I_N \Psi^\ld \|_{{X^{0,\frac{1}{2}-\eps}_T}}\Big\|_{L^{p}(\O)} 
\le C_p \jb{T}^\ta  N^{1-s}\ld^{-\frac 12 - s}\|\phi\|_{\HS(L^2;\dot{H}^s)}
\end{align}

\noi
 for any finite $ p\geq 1 $
 and $T > 0$, 
 where
 $\Psi^\ld$ is  the stochastic convolution corresponding to the
scaled noise $\phi^\ld \xi^\ld$.

Furthermore, if we assume
 $ \phi \in \HS(L^2,\dot{H}^\frac{3}{4})$, 
 then we have
\begin{align}
\|I_N \phi^\ld\|_{\HS(L^2;\dot{H}^\frac{3}{4})}\les \ld^{-\frac  54} 
\|\phi\|_{\HS(L^2;\dot{H}^\frac{3}{4})}, 
\label{sc-hs2}
\end{align}

\noi
uniformly in $N \geq 1$.

\end{lemma}

\begin{proof}
From 
\eqref{HS2a}, 
\eqref{I2}, and 
\eqref{HS4}, we have
\begin{align}
\begin{split}
\|I_N \phi^\ld\|_{\HS(L^2;\dot{H}^1)} 
& =  \|I_N k^\ld\|_{\dot{H}^1_xL^2_y}
\les N^{1-s}
\|\ld^{-2}
k(\ld^{-1} x,y )\|_{\dot{H}^s_xL^2_y}\\
& =  N^{1-s}\ld^{-\frac 12 - s} \|\phi\|_{\HS(L^2;\dot{H}^s)}.
\end{split}
\label{sc-hs3}
\end{align}

\noi
The second estimate \eqref{scale-sto-co}
follows from 
 Lemma \ref{LEM:stoconv} and \eqref{sc-hs}.
The last claim \eqref{sc-hs2} follows 
from proceeding as in \eqref{sc-hs3}
but using the uniform bound $|m_N(\xi)| \leq 1$ in the second step.
 \end{proof}

\begin{remark}\rm
(i) 
It is easy to check that Lemma \ref{LEM:scaling}
remains true even if we proceed  
with a scaling argument in \eqref{HS3} for any $a_2 \in \R$.

\smallskip

\noi
(ii) Lemma \ref{LEM:scaling} states that by choosing $\ld = \ld(N) \gg 1$, 
we can make the Hilbert-Schmidt norm of $I_N \phi^\ld$
arbitrarily small
(even after multiplying by $\ld T^\frac{1}{2}$; see \eqref{Y6} 
and \eqref{Y6a} below).

\end{remark}

We conclude this section by going over the scaling of the modified energy.
Let  $u_0^\ld = u^\ld(0)$.
Then, 
from \eqref{I2}, 
 the H\"ormander-Mikhlin multiplier theorem~\cite{G.L}, 
 \eqref{scaling}, and 
  Sobolev's inequality, 
we have 
\begin{align}
\begin{split}
E(I_Nu_{0}^{\ld})&=\frac{1}{2}\|\nb I_N u^\ld_0\|^2_{L^2}+\frac{1}{4}\|I_N u_0^\ld \|_{L^4}^4 \\
&\les N^{2-2s}\|u_0^\ld\|^2_{\dot{H}^s}+\ld^{-1} \|u_0\|^4_{L^4} \\
&\les N^{2-2s}\ld^{1-2s}\|u_0\|^2_{\dot{H}^s}+\ld^{-1}\|u_0\|^4_{H^s} \\
&\leq C_1N^{2-2s}\ld^{1-2s}\big(1+\|u_0\|_{H^s}\big)^4.
\end{split}
\label{scalENg}
\end{align}

\noi
Hence, for $\frac 12 < s < 1$, 
by choosing $\ld = \ld(N) \gg 1$, 
we can make the modified energy $E(I_Nu_0^\ld)$
of the scaled initial data 
arbitrarily small.

\section{On the commutator estimates}
\label{SEC:com}

In this section, we go over the commutator estimates
(Proposition \ref{PROP:CommBd}), 
corresponding to the deterministic component
in our application of the $I$-method. 

\begin{proposition}\label{PROP:CommBd}
Let $ \frac{5}{6}<s<1 $.
Then, given $\be > 0$, there exists small $\eps > 0$ such that 
\begin{align}
\bigg|\int_J\int_{\R^3}\cj{\Dl I_N u} [I_N, \NN](u)dxdt\bigg|
&\les  N^{-1+\be}
\|\nb I_N u\|_{X^{0,\frac{1}{2}-\eps}(J)}^4, 
\label{2D1}\\
\bigg|\int_J\int_{\R^3}\cj{I_N \NN(u)} [I_N, \NN](u) dxdt\bigg|&\les N^{-1+\be}
\|\nb I_N u\|_{X^{0,\frac{1}{2}-\eps}(J)}^6
\label{2D2}
\end{align}

\noi
for any $N\geq 1$ and any interval $J \subset [0, \infty)$,
where 
 $\NN(u) = |u|^2 u$ and 
the implicit constants are independent of $N \geq 1$
and $J \subset [0, \infty)$.
	
\end{proposition}

The estimates \eqref{2D1} and \eqref{2D2}
are essentially the same as   those appearing
in the proof of Proposition~4.1 in \cite{CKSTT0}.
The difference appears in the temporal regularity; 
on the right-hand sides of \eqref{2D1} and \eqref{2D2}, 
we have
$b = \frac 12 - \eps$, 
whereas the temporal regularity in \cite{CKSTT0} was $b = \frac 12 + \eps$.
The desired estimates in Proposition~\ref{PROP:CommBd}
follow
from the corresponding estimates in \cite{CKSTT0}
and an interpolation argument.

\begin{lemma}\label{LEM:Supportii}
\textup{(i)}
Let $u$ be a function on $\R \times \R^3$
with the spatial frequency support in  $\{|\xi|\sim M\}$ for some dyadic $M\geq 1$.
 Then, there exists $ \ta > 0$ such that 
\begin{align}
\|u\|_{L^{10}_tL^{10\pm}_x(J\times \R^3)}&\les  M^{\pm \theta}\|\nb u\|_{X^{0,\frac{1}{2}-}(J)},\label{S1}\\
\|u\|_{L^{\frac{10}{3}}_tL^{\frac{10}{3}-}_x(J\times \R^3)}&\les \|u\|_{X^{0,\frac{1}{2}-}(J)},\label{S2}\\
\|u\|_{L^{\frac{10}{3}}_tL^{\frac{10}{3}+}_x(J\times \R^3)}&\les  M^{\theta}\|u\|_{X^{0,\frac{1}{2}-}(J)}, \label{S3}
\end{align}

\noi
for any interval $J \subset [0, \infty)$,
where the implicit constants are independent of $N \geq 1$
and $J \subset [0, \infty)$.

\smallskip

\noi
\textup{(ii)} 
Let $ \frac{2}{3}<s<1 $. Then, the  following trilinear estimate holds:
\begin{align}
\|I_N (u_1 u_2 u_3)\|_{L^2_{t,x}(J\times\R^3)}
\lesssim \prod_{j=1}^{3}\|\nb I_N u_j\|_{X^{0,\frac{1}{2}-}(J)}
\label{S3a}
\end{align}

\noi
for any $N\geq 1$ and any interval $J \subset [0, \infty)$,
where the implicit constants are independent of $N \geq 1$
and $J \subset [0, \infty)$.

\end{lemma}

Note that $\frac 23 < \frac 56$.

\begin{proof}
Part (i) follows from (4.19), (4.20), and (4.21) in \cite{CKSTT0}, 
showing the corresponding estimates with $b = \frac 12 + \eps$
on the right-hand sides, 
and a simple interpolation argument.
From Sobolev's inequality and Lemma \ref{LEM:lin}\,(iii), we have 
\begin{align}
\| u \|_{L^{10}_t L^{10}_x(J \times \R^3)} 
\les \| \nb u \|_{L^{10}_t L^\frac{30}{13}_x(J \times \R^3)} 
\les \|\nb u \|_{X^{0, \frac 12 +}}.
\label{S3b}
\end{align}

\noi
Then, 
\eqref{S1} with the $+$ sign
follows from interpolating \eqref{S3b}
and 
\[ \| u \|_{L^{10}_t L^{10+}_x(J \times \R^3)} 
\les M^{\frac{1}{5}+} \|\nb u \|_{X^{0, \frac 25}}.\]

\noi
As for \eqref{S1} with the $-$ sign, 
interpolate \eqref{S3b}
with 
\[ \| u \|_{L^{10}_t L^{2}_x(J \times \R^3)} 
\les  M^{-1} \| \nb u \|_{X^{0, \frac 25}}.\]

\noi
As for \eqref{S2}, we interpolate \eqref{Str3}
with 
\[ \| u \|_{L^{\frac{10}{3}}_t L^{2}_x(J \times \R^3)} 
\les  \| u \|_{X^{0, \frac 15}},\]

\noi
while \eqref{S3} follows from Sobolev's inequality and \eqref{S2}.

Part (ii) corresponds to Lemma 4.3 in \cite{CKSTT0} but with $b = \frac 12-$.
By the interpolation lemma \cite[Lemma 12.1]{CKSTT2}, 
we may assume $N = 1$.
As in \eqref{HL}, write $u_j=u_j^\hi+u_j^{\lo}$.
As for $u_j^\lo$, we have
\begin{align}
\|u^\lo_j \|_{L^6_{t, x}}
\les \|\nb I u^\lo_j \|_{L^6_{t} L^2_x}
\les  \|\nb I u^\lo_j \|_{X^{0, \frac{1}{3}}}.
\label{S4}
\end{align}

\noi
As for $u_j^\hi$, noting $I\sim |\nb |^{s-1}$, we have
\begin{align*}
\|u^\hi_j \|_{L^6_{t, x}}
& \sim \big\| |\nb|^{1-s} I u^\hi_j \big\|_{L^6_{t, x}}
\les \big\| |\nb|^{\frac 53-s} I u^\hi_j \big\|_{L^6_{t}L^\frac{18}{7}_x}\\
& \les \big\| \jb{\nb}^{1-} I u^\hi_j \big\|_{X^{0, \frac{1}{2}+}},  
\end{align*}

\noi
\noi
provided that $s > \frac 23$, 
where we used Lemma \ref{LEM:lin}\,(iii) in the last step.
Interpolating this with 
\begin{align*}
\|u^\hi_j \|_{L^6_{t, x}}
& \sim \big\| |\nb|^{1-s} I u^\hi_j \big\|_{L^6_{t, x}}
\les \big\| |\nb|^{2-s} I u^\hi_j \big\|_{L^6_{t}L^2_x}\\
& \les \big\| \jb{\nb}^{2-s} I u^\hi_j \big\|_{X^{0, \frac{1}{3}}},  
\end{align*}

\noi
we obtain 
\begin{align}
\|u^\hi_j \|_{L^6_{t, x}}
& \les \big\| \jb{\nb} I u^\hi_j \big\|_{X^{0, \frac{1}{2}-}}
\sim \|\nb I u^\hi_j \|_{X^{0, \frac{1}{2}-}}.
\label{S5}
\end{align}

\noi
Then, \eqref{S3a} follows from 
the boundedness of $m_1(\xi)$
and 
$L^6_{t, x}, L^6_{t, x}, L^6_{t, x}$-H\"older's inequality with~\eqref{S4} and~\eqref{S5}.
\end{proof}

We now briefly discuss a proof of  Proposition~\ref{PROP:CommBd}.

\begin{proof}[Proof of Proposition~\ref{PROP:CommBd}]

The estimates \eqref{2D1} and \eqref{2D2} (with $b = \frac 12-$) follow
from a small modification of the proof of Proposition 4.1 in \cite{CKSTT0} (with $b = \frac 12+$), 
using 
 Lemma~\ref{LEM:Supportii}.
In the following, we present 
the proof of  the second estimate \eqref{2D2}.
As for the first estimate~\eqref{2D1}, we
briefly discuss the required modifications 
at the end of the proof.
%
%
%
%
In the following, we fix $N$ and drop the subscript $N$ from $I_N$ and $m_N$.

From the definition \eqref{Ix2} of the $I$-operator, we can rewrite the left-hand side of \eqref{2D2}
as 
\begin{align}
\bigg| \int_J \intt_{\xi_1+\xi_2+ \xi_3+\xi_4= 0}
M_N(\bar \xi)  \ft{\cj{I \NN(u)} } (\xi_{1}, t) 
\ft{Iu}(\xi_2, t)\ft{\cj{Iu}}(\xi_3, t)\ft{Iu}(\xi_4, t)
d\xi_1d\xi_2d\xi_3 dt \bigg|, 
\label{Z1}
\end{align}

\noi
where the multiplier $M_N (\bar \xi)$ is given by 
\begin{align}
M_N(\bar \xi) = 
M_N(\xi_1, \xi_2, \xi_3, \xi_4)
= 
\frac{m(\xi_{1})}{m(\xi_2)m(\xi_3)m(\xi_4)}
- 1.
\label{Z2}
\end{align}

\noi
We suppress the $t$-dependence in the following.

\noi

With the Littlewood-Paley projector $\P_{N_j}$ in \eqref{Px1}, 
we have 
\begin{align}
\eqref{Z1} \le
\sum_{\substack{N_1, \dots, N_4 \geq 1\\\text{dyadic}}}\bigg| \int_J 
\intt_{\xi_1+\xi_2+ \xi_3+\xi_4= 0}
M_N(\bar \xi)  \ft{U}_1 (\xi_1) 
\prod_{j = 2}^4 \ft{u}_j(\xi_j) 
d\xi_1d\xi_2d\xi_3 dt \bigg|, 
\label{Z3}
\end{align}

\noi
where  
$U_1=  \cj{ \P_{N_1} I \NN(u)}$, 
$u_j =  \P_{N_j} I u$, $j = 2, 4$, 
and $u_3 = \cj{\P_{N_3}I u}$.
Without loss of generality, we assume that $N_2 \ge N_3 \ge N_4$.
Note that  we  have $N_1 \les N_2$
under $\xi_1+\xi_2+\xi_3+\xi_4 = 0$ and $|\xi_j| \sim N_j$.
Thus,  if we have $N_2 \ll N$ in addition, 
then it follows from \eqref{Ix1} and \eqref{Z2} that  
$M_N(\bar \xi) =  0$.
Therefore, we assume that $N_2 \ges N$ in the following.
We also note that under $N_2 \ges N_1$ and $N_2 \ge N_3 \ge N_4$, 
we have 
\begin{align}
m(N_2) m(N_3) m(N_4) 
\le m(N_2)  \les m(N_1)
\label{Z4}
\end{align}

\noi
Then, from \eqref{Z2} and \eqref{Z4} with \eqref{Ix1}, we have 
\begin{align}
\begin{split}
\frac{M_N(N_1, N_2, N_3, N_4)}{N_2^{1-\eps}} 
& \sim \frac{m(N_1)}{N_2^{1-\eps} \prod_{j = 2}^4 m(N_j)}
 \le 
\frac{1}{N_2^{1-\eps} (m(N_2))^3}\\
& \leq N^{-1+2\eps} N_2^{-\eps} \bigg(\frac{N}{N_2}\bigg)^{1-2\eps - 3(1-s)}\\
& \les N^{-1+2\eps} N_2^{-\eps} 
\end{split}
\label{Z5}
\end{align}

\noi
for any sufficiently small $\eps > 0$, 
provided that $s > \frac {2 + 2\eps}{3}$.

\smallskip

\noi
$\bullet$
{\bf Case 1:}
 $N_j \ges 1$, $j = 1, \dots, 4$.
\\
\indent
In this case,
by  applying 
the $L^2_{t, x},  L^{\frac{10}{3}}_{t} L^{\frac{10}{3}+}_{x}, 
 L^{10}_{t} L^{10-}_{x},  L^{10}_{t} L^{10-}_{x}$-H\"older's inequality\footnote{It is understood that the time integrations
 are restricted to the interval $J$.
 The same comment applies in the following.}
 and \eqref{Z5} to \eqref{Z3}
 and then applying Lemma \ref{LEM:Supportii}, we obtain
\begin{align*}
\eqref{Z1} 
& \les N^{-1+} 
\sum_{\substack{N_1, \dots, N_4 \geq 1\\\text{dyadic}}}
\bigg(\prod_{j = 1}^4 N_j^{-\frac{\eps}{4}}\bigg)
\| U_1 \|_{L^2_{t, x}}
N_2^{1-} \| u_2 \|_{L^{\frac{10}{3}}_{t} L^{\frac{10}{3}+}_{x}}
\prod_{j = 3}^4 \|u_j\|_{ L^{10}_{t} L^{10-}_{x}}\\
& \les N^{-1+} \prod_{j = 1}^6 
 \|\nb I u_j\|_{X^{0,\frac{1}{2}-}(J)}
\end{align*}

\smallskip

\noi
$\bullet$
{\bf Case 2:}
 $N_1 \ges 1 \gg N_4$.
\\
\indent
We proceed as in Case 1
but 
 we place $u_j$, $j = 3, 4$, 
 in 
 the $ L^{10}_{t} L^{10+}_{x}$-norm
 when  $N_j \ll1$.
In view of
Lemma \ref{LEM:Supportii}\,(i), 
this creates 
a small positive power of $N_j$, allowing 
us to sum over dyadic $N_j \ll1$.

For $N_3 \ges 1$, 
we 
apply 
the $L^2_{t, x},  L^{\frac{10}{3}}_{t} L^{\frac{10}{3}+}_{x}, 
 L^{10}_{t} L^{10-}_{x},  L^{10}_{t} L^{10+}_{x}$-H\"older's inequality
 and proceed as in Case 1.
For $N_3 \ll 1$, 
we 
apply 
the $L^2_{t, x},  L^{\frac{10}{3}}_{t} L^{\frac{10}{3}-}_{x}, 
 L^{10}_{t} L^{10+}_{x},  L^{10}_{t} L^{10+}_{x}$-H\"older's inequality
  and proceed as in Case 1.

\smallskip

\noi
$\bullet$
{\bf Case 3:}
 $N_1 \ll 1$.
\\
\indent
In this case, we use 
\begin{align}
\|U_1 \|_{L^2_t L^{2+}_x}
\les N_1^{0+}
\|U_1 \|_{L^2_{t, x}}
\label{Z6}
\end{align}

\noi
to gain a small power of $N_1$.
For $N_4 \ges 1$, 
we 
apply 
the $L^2_{t}L^{2+}_x,  L^{\frac{10}{3}}_{t} L^{\frac{10}{3}+}_{x}, 
 L^{10}_{t} L^{10-}_{x},  L^{10}_{t} L^{10-}_{x}$-H\"older's inequality
 and proceed as in Case 1
 with  Lemma \ref{LEM:Supportii} and \eqref{Z6}.

For $N_3 \ges 1\gg N_4$, 
we 
apply 
the $L^2_{t}L^{2+}_x,  L^{\frac{10}{3}}_{t} L^{\frac{10}{3}+}_{x}, 
 L^{10}_{t} L^{10-}_{x},  L^{10}_{t} L^{10+}_{x}$-H\"older's inequality.
 For $N_3, N_4 \ll1 $, 
we 
apply 
the $L^2_{t}L^{2+}_x,  L^{\frac{10}{3}}_{t} L^{\frac{10}{3}-}_{x}, 
 L^{10}_{t} L^{10+}_{x},  L^{10}_{t} L^{10+}_{x}$-H\"older's inequality.
Then, 
proceeding as in Case 1
 with  Lemma \ref{LEM:Supportii} and \eqref{Z6}, 
 we obtain the desired estimate~\eqref{2D2}.

As for the first estimate \eqref{2D1}, 
we can simply repeat the argument in \cite{CKSTT0}
with Lemma~\ref{LEM:Supportii}\,(i) in place 
of \cite[(4.19), (4.20), and (4.21)]{CKSTT0}
and replacing $L^{\frac{10}{3}}_{t, x}$ by $L^{\frac{10}{3}}_{t}L^{\frac{10}{3}-}_{x}$
(so that \eqref{S2} with $b = \frac 12-$ is applicable).
We point out that the regularity restriction $s > \frac 56$ comes from this part.
\end{proof}

\section{On growth of the modified energy}\label{SEC:energy}

In this section, 
we use stochastic analysis to study  growth of  
the modified energy $E(I_Nu)$ associated with $I$-SNLS 
\eqref{SNLSI}. 
Before doing so, 
 we first go over the  deterministic setting. 
 Given a smooth solution $u$ to the cubic NLS \eqref{NLS}, 
 we can verify the conservation of the 
 energy $ E(u)=\frac{1}{2}\int_{\R^3} |\nb u|^2 dx +\frac{1}{4}\int_{\R^3}|u|^4 dx $ 
  by simply differentiating in time
 and using the equation  \eqref{NLS}:
\begin{align*}
\dt E(u(t)) &= \Re   \int_{\R^3} \cj{ \dt u} (|u|^2u-\Dl u) dx  \\
&=\Re   \int_{\R^3} \cj{ \dt u} (|u|^2u-\Dl u- i\dt u) dx  =0.
\end{align*}

\noi
In a similar manner, 
given a smooth solution $u$ to the cubic NLS \eqref{NLS}, 
the time derivative of the modified energy $E(I_Nu)$
is given by 
\begin{align}\label{Igrowth}
\begin{split}
\dt E(I_N u (t) )
&= \Re   \int_{\R^3} \cj{ \dt I_N u} \big(|I_N u|^2 I_N u - I_N(|u|^2u)\big) dx  \\
&=  - \Re     \int_{\R^3} \cj{\dt I_N u} [I_N, \NN](u)     dx   .
\end{split}
\end{align}

\noi
Then, the fundamental theorem of calculus
yields
\begin{align}
|  E(I_N u(t_2))- E(I_N u(t_1))   |=\bigg|    \Re \int_{t_1}^{t_2} \int_{\R^3} \cj{\dt I_N u} [I_N, \NN](u)dxdt  
 \bigg|
\label{Igrowth2}
\end{align}

\noi
for any $t_1, t_2 \in \R$, 
where the right-hand side can be estimated by the commutator estimate;
see \cite[Proposition 4.1]{CKSTT0}.

For our problem. 
we need to  estimate growth of the modified energy $E(I_N u)$,
where $u$ is now a solution to  SNLS \eqref{SNLS0}
with a stochastic forcing.\footnote{As we see in Section \ref{SEC:end}, 
we in fact apply the argument in this section 
after applying a suitable scaling.} 
As such, 
we need to proceed with Ito's lemma.

\begin{lemma}\label{LEM:Ito}
Given $N \geq 1$, let $I_Nu$ be the solution to $I$-SNLS \eqref{SNLSI}
with $I_N u|_{t = 0} = I_N u_0$, 
where 
$\phi$ and $u_0$ are as in  Proposition \ref{PROP:LWP}.
Moreover, 
given $T >0 $, let 
$\tau$ be a stopping time 
with $ 0<\tau<\min(T^{*}, T) $ almost surely, 
where $ T^* = T^*_\o$ is the (random) forward maximal time of existence
for $I$-SNLS \eqref{SNLSI}, satisfying \eqref{LWP2}.
Then, we have 
\begin{align}
\begin{split}
E(I_N u(\tau))
&=E(I_N u_0)\\
&\phantom{=}
-\Im\int_{0}^{\tau}\int_{\R^3}
\cj{
\Dl I_N u - I_N \NN(u)}
 \, [I_N, \NN](u)dxdt
\\ 
&\phantom{=}-\Im\sum_{n \in \N}\int_{0}^{\tau}\int_{\R^3}
\cj{\Dl I_N u}I_N \phi e_ndxd\b_n(t)
\\
&\phantom{=}+ \Im \sum_{n \in \N}\int_0^{\tau}\int_{\R^3}
\cj{\NN(I_N u)} I_N \phi e_n dxd\b_n(t)
\\ 
&\phantom{=}+2 \sum_{n \in \N} \int_0^{\tau}\int_{\R^3}|I_N u I_N \phi e_n|^2dxdt 
\\ 
&\phantom{=}+ \tau\|I_N \phi\|_{\HSii}^2. 
\end{split}
\label{Ito1}
\end{align}

\noi
Furthermore, 
if we assume that 
$I_N \phi \in \HS(L^2;\dot{H}^\frac{3}{4})$ in addition, then 
there exists $ C>0 $ such that 
\begin{align}
\begin{split}
\E\bigg[ \sup_{0\leq t\leq \tau} E(I_N u(t))  \bigg]
&\leq 2E(I_N u_0)
+CT\|I_N \phi\|^{2}_{\HS(L^2;\dot{H}^1)}
+ C T^2 
\|I_N \phi\|^4_{\HS(L^2;\dot{H}^\frac{3}{4})}
\\
&\hphantom{X}
+C
\E\Bigg[ \sup_{0\leq t\leq \tau}
\bigg|\Im \int_{0}^t\int_{\R^3} 
\cj{
\Dl I_N u - I_N \NN(u)}
\, [I_N, \NN](u)dxdt' \bigg|\Bigg].
\end{split}
\label{EngBc}
\end{align}

\end{lemma}

\begin{remark}\label{REM:martin}\rm
(i) 
The second term on the right-hand side of \eqref{Ito1} 
corresponds to the contribution 
from the commutator term $[I_N, \NN]$, also present in the deterministic case.
The remaining terms 
are the additional terms, appearing from the application of Ito's lemma.
Part (ii) follows from \eqref{EngBc}
and  Burkholder-Davis-Gundy  inequality.
We need to hide some terms on the right-hand side of \eqref{Ito1}
to the left-hand side, which results in 
 the factor of 2 appears in the first term 
on the right-hand side of \eqref{EngBc}.

In the deterministic setting, one can 
apply the commutator estimate to \eqref{Igrowth2}
on each local-in-time interval.
In the current stochastic setting, 
however, 
it is not possible to apply the estimate \eqref{EngBc}
(and the commutator estimates (Proposition \ref{PROP:CommBd}))
on each local-in-time interval
since
this factor of 2 on $E(I_Nu_0)$
in \eqref{EngBc} would lead to an exponential growth of the constant
in iterating the local-in-time argument.

\smallskip

\noi
(ii) Our convention states that $\be_n$ in \eqref{Psi1} is a complex-valued Brownian motion.
This is the reason we do not have a factor $\frac 12$ on the last term in \eqref{Ito1}.

\smallskip

\noi
(iii) In controlling 
  the  fourth and fifth terms
on the right-hand side of \eqref{Ito1}, 
we need to use
the $\HS(L^2;\dot{H}^\frac{3}{4})$-norm
of $I_N \phi$ in \eqref{EngBc}.

\end{remark}

\begin{proof}

A formal application of Ito's lemma yields \eqref{Ito1}.
One can justify the computation by inserting 
several truncations
and the local well-posedness argument.  See \cite{DD2} for details
when there is no $I$-operator.

Let us now turn to Part (ii).
By 
Burkholder-Davis-Gundy inequality
(\cite[Theorem 3.28 on p.\,166]{KS}), 
 Cauchy-Schwarz inequality,  and Cauchy's inequality, we 
estimate the third term on the right-hand side of \eqref{Ito1}
as
\begin{equation}
\begin{split}
\E\Bigg[ & \supT
 \bigg(\Im\sum_{n\in\N}\int_{0}^{t}\int_{\R^3}\cj{\Dl I_N u} I_N \phi e_n   dx d\beta_n(t')\bigg)\Bigg]\\
&\le C \, \E\Bigg[\bigg(\sum_{n\in\N}\int_{0}^{\tau}\bigg|\int_{\R^3}\nb I_N u\cdot \nb I_N \phi e_ndx\bigg|^2dt  \bigg)^\frac12\Bigg]\\
&\le C T^\frac12\|I_N \phi \|_{\HSii} \E\bigg[ \supT \|\nb I_N u(t) \|_{L^2}   \bigg]\\
&\le C T 
\|I_N \phi\|^2_{\HS(L^2;\dot{H}^1)}+\frac{1}{8}\E\bigg[\supT E(I_N u(t))\bigg].
\end{split}
\label{P1}
\end{equation}

\noi
By Burkholder-Davis-Gundy inequality
and Sobolev's inequality $\dot H^\frac{3}{4}(\R^3) \subset L^4(\R^3)$, 
the fourth terms is estimated as 
\begin{equation}
\begin{split}
\E\Bigg[ & \supT\bigg(\Im\sum_{n\in\N}\int_{0}^{t}\int_{\R^3}|I_N u|^2\cj{I_N u} I_N \phi e_ndxd\beta_n(t')\bigg)\Bigg]\\
&\le
C\E\Bigg[\bigg(\sum_{n\in\N}\int_{0}^{\tau}\bigg|\int_{\R^3}|I_N u|^2\cj{I_N u} I_N \phi e_ndx\bigg|^2dt\bigg)^{\frac12}\Bigg]\\
&\le C T^\frac12\|I_N \phi \|_{\HS(L^2\dot H^\frac{3}{4})} 
\E\bigg[ \supT \| I_N u(t) \|_{L^4}^3   \bigg]\\
&\le C T^2 
\|I_N \phi\|^4_{\HS(L^2;\dot{H}^\frac{3}{4})}+\frac{1}{8}\E\bigg[\supT E(I_N u(t))\bigg].
\end{split}
\label{P2}
\end{equation}

\noi
By Sobolev's inequality, 
we estimate the fifth term as 
\begin{align}
\begin{split}
2 \E\bigg[ & \sum_{n\in\N}\int_0^\tau \int_{\R^3}|I_N u I_N \phi e_n|^2dxdt\bigg]\\
& \le C T
\|I_N \phi \|_{\HS(L^2\dot H^\frac{3}{4})}^2 
\E\bigg[ \supT \| I_N u(t) \|_{L^4}^2   \bigg]\\
&\le C T^2 
\|I_N \phi\|^4_{\HS(L^2;\dot{H}^\frac{3}{4})}+\frac{1}{8}\E\bigg[\supT E(I_N u(t))\bigg].
\end{split}
\label{P3}
\end{align}

Finally, 
the desired estimate \eqref{EngBc} follows from 
\eqref{Ito1}, 
\eqref{P1}, \eqref{P2}, and \eqref{P3}.
\end{proof}

\section{Proof of Theorem \ref{THM:main}}
\label{SEC:end}

In this section, 
we present global well-posedness 
of SNLS \eqref{SNLS0} (Theorem \ref{THM:main}). 
In the current stochastic setting, 
it suffices
to prove the following
``almost'' almost sure global well-posedness.

\begin{proposition}\label{PROP:aas}
Given  $\frac56< s < 1$, 
let  $u_0 \in H^s(\R^3)$ and  $\phi \in \HS(L^2; H^s)$.
Then, given any $T, \eps > 0$, there exists a set $ \O_{T, \eps}\subset \O$
such that

\smallskip
\noi
\textup{(i)}
$P( \O_{T, \eps}^c) < \eps$.

\smallskip
\noi
\textup{(ii)}
For each $\o \in  \O_{T, \eps}$, there exists a \textup{(}unique\textup{)} solution $u$
to SNLS \eqref{SNLS0} in $C([0,T];H^s(\R^3)) $
with $u|_{t = 0} = u_0$
and the noise given by $\phi \xi = \phi \xi(\o)$.

\end{proposition}

Once we prove  Proposition \ref{PROP:aas}, 
Theorem \ref{THM:main} follows
from the Borel-Cantelli lemma.
See, for example,  \cite{CO, BOP2}.
See also Remark \ref{REM:growth}.
Hence, in the remaining part of this paper, 
we focus on proving Proposition~\ref{PROP:aas}.

\begin{proof}[Proof of Proposition \ref{PROP:aas}]

As in the deterministic setting \cite{CKSTT0}, 
we first apply the scaling~\eqref{scaling}, 
where 
 $\ld  = \ld (N) = \ld(T, \eps)\gg1 $ is to be chosen later.
Note that, given $\o \in \O$,  
 $ u = u(\o) $ solves \eqref{SNLS0} on $[0, T]$
 if and only if the scaled function $ u^\ld = u^\ld(\o)$ solves \eqref{SNLS0} 
 on $[0, \ld^2 T]$ with the scaled initial data $ u_0^\ld = u^\ld(0)$.
We then apply the $I$-operator to the scaled function $u^\ld$.
In the following, we focus on studying the scaled $I$-SNLS \eqref{SNLSI2}.
In view of 
 Remark~\ref{REM:L2bound} 
 and~\eqref{I2}
 with Lemma \ref{LEM:bound}, it suffices to show that 
 $ \|I_N \uu\|_{\dot{H}^1} $ remains finite  on $[0, \ld^2 T]$
 with a large probability.

Fix   $\frac 56 < s < 1$ and $u_0 \in H^s(\R^3)$.
Given large $T \gg 1 $ and small $\eps > 0$, fix $N= N(T, \eps ) \gg1 $ (to be chosen later).
We now choose $\ld = \ld(N) \gg 1$ such that 
\begin{align}
 N^{2-2s}\ld^{1-2s} \ll N^{-2\ta}
\label{Y0}
\end{align}

\noi
for some small $\ta >0$.
More precisely, 
we can choose
\begin{align}
\ld \sim  N^\frac{2- 2s+2\ta }{2s-1} \gg 1
\label{Y2}
\end{align}

\noi 
under the condition  that $\frac 12 < s < 1$.
Then, from the scaling property 
\eqref{scalENg} of the modified energy
and~\eqref{Y0}, we have 
\begin{align}
E(I_N u_{0}^{\ld})
\leq C(u_0) N^{2-2s}\ld^{1-2s} 
\ll N^{-2\ta}  \eta_0 
\ll \eps \eta_0
\label{Y1}
\end{align}

\noi
by choosing $N = N(\eps) \gg 1$,  
where $\eta_0$ is as in 
\eqref{ZX4}.

Let $\Psi^\ld$ denote the stochastic convolution corresponding to the
scaled noise $\phi^\ld \xi^\ld$.
By Lemma \ref{LEM:scaling} with $p \gg 1$,  \eqref{Y0}, \eqref{Y2}, 
and choosing $N = N(T, \eps) \gg 1$, we have, for $p \geq 2$, 
\begin{align}
\begin{split}
\Big\| \|\nb I_N \Psi^\ld \|_{X^{0,\frac{1}{2}-}{([j,j+1])}}
\Big\|_{L^p(\O)}
&\les N^{ - \ta} \ld^{-1} \|\phi \|_{\HS(L^2; \dot H^s)}\\
&\ll (\eps \ld^{-2} T^{-1})^\frac{1}{p} \eta_1,
\end{split}
\label{stc}
\end{align}

\noi
uniformly in  $j \in \N \cup \{0\}$, 
where $\eta_1$ is as in \eqref{ZX4}.
Note that by choosing $p \gg1 $, 
\eqref{stc} imposes only a mild condition
$N \geq (\eps^{-1} T)^{0+}$.
For $j \in \N \cup \{0\}$, 
define $A^j_\eps \subset \O$ by 
\begin{align}
A^j_\eps =\bigg\{    \o\in\O:   \|\nb I_N \Psi^\ld   \|_{X^{0,\frac{1}{2}-}_{[j ,j + 1 ]}}\leq \eta_1   \bigg\}.
\label{Y3}
\end{align}

\noi
Now, set $\O_{T, \eps}^{(1)}$ by setting
\begin{align*}
\O^{(1)}_{T, \eps}=\bigcap_{j=0}^{[ \ld^2T ]} A^j_\eps, 
\end{align*}

\noi
where $[ \ld^2T ]$ denotes the integer part of $\ld^2 T$.
Then, it follows
from Chebyshev's inequality
and \eqref{stc} that 
\begin{align}
P\big( \O\setminus \O^{(1)}_{T, \eps} \big) 
<  \frac \eps 2.
\label{Y5}
\end{align}

Lastly, note that from 
Lemma \ref{LEM:scaling} with \eqref{Y0} and \eqref{Y2}, we have
\begin{align}
\begin{split}
\ld T^\frac{1}{2} \|I_N \phi^\ld\|_{\HS(L^2;\dot{H}^1)}
& \les N^{1-s}\ld^{\frac 12 - s} T^\frac{1}{2}\|\phi\|_{\HS(L^2;\dot{H}^s)}
\ll N^{-\ta} T^\frac{1}{2}\|\phi\|_{\HS(L^2;H^s)}\\
& \ll \eps^\frac{1}{2} \|\phi\|_{\HS(L^2;H^s)}
\end{split}
\label{Y6}
\end{align}

\noi
and
\begin{align}
\begin{split}
\ld T^\frac{1}{2} \|I_N \phi^\ld\|_{\HS(L^2;\dot{H}^\frac{3}{4})}
& \les \ld^{- \frac 14} T^\frac{1}{2}\|\phi\|_{\HS(L^2;\dot{H}^\frac{3}{4})}
\ll N^{-\g} T^\frac{1}{2}\|\phi\|_{\HS(L^2;H^s)}\\
& \ll \eps^\frac{1}{4} \|\phi\|_{\HS(L^2;H^s)}
\end{split}
\label{Y6a}
\end{align}

\noi
by choosing $N = N(T, \eps) \gg 1$ sufficiently large, 
where $\g$ is given by 
\begin{align*}
\g = \frac{1-s+\ta}{4s-2}>0.
\end{align*}

Now, we define a stopping time $\tau$ by 
\begin{align}\label{stop-t}
\tau  =\tau _\o (I_N \uu_0, I_N \phi^\ld) 
: =\inf\Big\{   t:   
\sup_{0\leq t' \leq t} E(I_N u^\ld(t'))\geq \eta_0 \Big\},
\end{align}

\noi
where $\eta_0$ is as in \eqref{ZX4}.
Note that in view of the blowup alternative
stated in Proposition~\ref{PROP:LWP},  
the condition \eqref{stop-t}
guarantees that the solution 
$I_N u^\ld$ to the scaled $I$-SNLS \eqref{SNLSI2}
exists on $[0, \tau]$.
Then, set 
\begin{align}
 \O^{(2)}_{T, \eps} 
= \Big\{   \o\in \O : \tau \geq \ld^2 T \Big\}
\label{Y7a}
\end{align}

\noi
and 
\begin{align}
 \O_{T, \eps} = 
\O^{(1)}_{T, \eps} \cap \O^{(2)}_{T, \eps} .
\label{Y8}
\end{align}

We claim that
\begin{align}
P\big( \O_{T, \eps}^{(1)} \setminus \O^{(2)}_{T, \eps} \big) < \frac \eps2.
\label{Y9}
\end{align}

\noi
Then, it follows from \eqref{Y8} with \eqref{Y5} and \eqref{Y9}
that 
\begin{align}
P( \O_{T, \eps}^c) < \eps.
\label{Y10}
\end{align}

In the following, we prove \eqref{Y9}.
Let $ \o\in \O_{T, \eps}^{(1)} \setminus \O^{(2)}_{T, \eps}$.
Then, from Remark  \ref{REM:local} 
with \eqref{Y3}
and \eqref{stop-t}, 
we have 
\[
\|\nb I_N \uu\|_{X^{0,\frac{1}{2}-}([j,j+1])} \leq C_0
\]

\noi
for any $j \in \N \cup \{0\}$ such that $j + 1\leq \tau$.
Hence, it follows from Proposition \ref{PROP:CommBd} 
with \eqref{SNLSI2} that 
\begin{align}\label{commu}
\bigg| \Im \int_{j}^{j+1}\int _{\R^3} 
\cj{
\Dl I_N \uu - I_N \NN( \uu )}
[I_N, \mathcal N](\uu)dxdt \bigg| \les N^{-1+}.
\end{align}

\noi
Then, from Lemma \ref{LEM:Ito} and \eqref{commu}, we have
one can write \eqref{EngBc} as:
\begin{align}
\begin{split}
\E& \bigg[\sup_{0\leq t\leq \tau \wedge \ld^2T}   E(I_N \uu(t))\bigg]  \\
& \leq 2E(I_N \uu_0)
+C \ld^2T\|I_N \phi^\ld\|^2_{\HSii}
+ C \ld^4T^2 \|I_N \phi^\ld\|^4_{\HS(L^2;\dot{H}^\frac{3}{4})}
\\
& \hphantom{X}
+C \ld^2 TN^{-1+}.
\end{split}
\label{exbound}
\end{align}

On the other hand, from \eqref{stop-t} and the continuity
of the modified energy (in time), 
we have 
\begin{align}
\sup_{0\leq t\leq \tau\wedge \ld^2 T}   E(I_N \uu(t;\o))
=\eta_0 
\label{Y11}
\end{align}

\noi
for any 
$ \o\in \O_{T, \eps}^{(1)} \setminus \O^{(2)}_{T, \eps}$.
Hence, 
from \eqref{exbound} and \eqref{Y11}
with \eqref{Y1},  \eqref{Y6}, and  \eqref{Y6a},we have
\begin{align}
\begin{split}
P\big( \O_{T, \eps}^{(1)} \setminus \O^{(2)}_{T, \eps} \big) 
& = \E\Big[ \ind_{\O_{T, \eps}^{(1)} \setminus \O^{(2)}_{T, \eps}}\Big]\\
& =  \eta_0^{-1} 
\E\bigg[
\ind_{\O_{T, \eps}^{(1)} \setminus \O^{(2)}_{T, \eps}}\cdot
\sup_{0\leq t\leq \tau \wedge \ld^2T}   E(I_N \uu(t))\bigg]  \\
& \leq \eta_0^{-1} 
\E\bigg[\sup_{0\leq t\leq \tau \wedge \ld^2T}   E(I_N \uu(t))\bigg]  \\
& \leq 2\eta_0^{-1} E(I_N \uu_0)+C\eta_0^{-1}  \ld^2T\|I_N \phi^\ld\|^2_{\HSii}
\\
& \hphantom{X} 
+ C \eta_0^{-1} \ld^4T^2 \|I_N \phi^\ld\|^4_{\HS(L^2;\dot{H}^\frac{3}{4})}
+C \eta_0^{-1} \ld^2 TN^{-1+}\\
& \leq \frac{\eps}{4}
+C \eta_0^{-1} \ld^2 TN^{-1+}.
\end{split}
\label{Y12}
\end{align}

\noi
As in the deterministic case \cite{CKSTT0}, 
we can make the last term on the right-hand side of \eqref{Y12}
small, provided that $s > \frac 56$.
In fact, with \eqref{Y2}, we can choose $N = N(T, \eps) \gg 1$ such that 
\begin{align}\label{comm-small}
T\les  \eps \ld^{-2}N^{1-}
&\sim \eps N^{\frac{6s-5 -4\ta -}{2s-1}}, 
\end{align}

\noi 
guaranteeing 
\begin{align}
C \eta_0^{-1} \ld^2 TN^{-1+} <  \frac{\eps}{4}.
\label{Y13}
\end{align}

\noi
Note that \eqref{comm-small} is possible
only when $6s > 5 + 4\ta$, 
which can be satisfied when $s > \frac 56$
by choosing $\ta = \ta(s) > 0$ sufficiently small.

Therefore, the desired bound \eqref{Y9} follows
from \eqref{Y12} and \eqref{Y13}, 
and thus \eqref{Y10} holds
by choosing
$N = N(T, \eps) \gg 1$ such that 
\eqref{Y1}, \eqref{stc}, \eqref{Y6},  \eqref{Y6a}, 
and \eqref{comm-small}
are satisfied.

By  the definition \eqref{Y8}, 
for any $\o \in \O_{T, \eps}$, 
the solution $I_Nu^\ld = I_N u^\ld(\o)$ to 
the scaled $I$-SNLS \eqref{SNLSI2}
exists on the time interval $[0, \ld^2T]$.
Together with \eqref{Y10}, this proves Proposition \ref{PROP:aas}.
\end{proof}

\begin{remark}\label{REM:growth}
\rm

As in the usual application of the $I$-method
in the deterministic setting, our proof of Proposition \ref{PROP:aas}
yields a polynomial growth bound
on the $H^s$-norm of a solution.

From the scaling \eqref{scaling}, 
we have
\begin{align}
E(I_N u^\ld (\ld^2t))=\ld^{-1}    E(I_N u(t)).
\label{X1}
\end{align}

\noi
Thus,  
given $T>0$, from \eqref{X1}, 
\eqref{stop-t}, \eqref{Y7a}, \eqref{Y8}, \eqref{Y2}, 
and \eqref{comm-small}, we have
\begin{align}
\begin{split}
E(I_N u(T))& = \ld E(I_N u^\ld (\ld^2T))    
\leq \ld \eta_0 \les \ld\\
&\les N^{\frac{2-2s+2\theta}{2s-1}}\\
&\les T^{\frac{2-2s+2\theta}{6s-5-4\theta-}}=T^{\frac{1-s-\theta}{3(s-\frac 56)-2\theta-}} 
\end{split}
\label{X2}
\end{align}

\noi
for $\o \in \O_{T, \eps}$, 
where the implicit constant depends on $u_0$ and $\O_{T, \eps}$
On the other hand, 
from~\eqref{I2}, we have 
\begin{align}
\|u(t)\|^2_{H^s}&\les E(I_N u(t))+\|u(t)\|^2_{L^2}.
\label{X3}
\end{align}

\noi
Lastly, 
it follows 
Remark \ref{REM:L2bound} (in particular, Footnote \ref{FT1})
that 
\begin{align}
\sup_{0\le t \le T}\|u(t)\|^2_{L^2} \le C(u_0, \phi, \o) T.
\label{X4}
\end{align}

\noi
Therefore, from \eqref{X2}, \eqref{X3}, and \eqref{X4}, 
we conclude that 
\begin{align}
\|u(t)\|^2_{H^s}\le C(u_0, \phi, \o) \max\Big(T^{\frac{1-s-\theta}{3(s-\frac 56)-2\theta-}} , T\Big)
\label{X5}
\end{align}

\noi
for any $0 \le t \le T$ and 
 $\o \in \O_{T, \eps}$.

Let $u_0 \in H^s(\R^3)$ for some $s > \frac 56$.
Given small $\eps > 0$, 
we apply Proposition \ref{PROP:aas}
and construct a set $\O_{2^j, 2^{-j}\eps}$
for each $j \in \N$. 
Now, set $\Si  = \bigcup_{0 < \eps \ll 1} \bigcap_{j \in \N} \O_{2^j, 2^{-j}\eps}$.
Then, for each $\o \in \Si$, 
there exists $\eps > 0$ 
such that $\o \in  \bigcap_{j \in \N} \O_{2^j, 2^{-j}\eps}$.
In particular, the corresponding solution $u$ to \eqref{SNLS}
exists globally in time.
Furthermore, from \eqref{X5}, we have 
\begin{align*}
\|u(t)\|^2_{H^s}\le C(u_0, \phi, \o) \max\Big(t^{\frac{1-s-\theta}{3(s-\frac 56)-2\theta-}} , t\Big)
\end{align*}

\noi
for any $t > 0$.


\end{remark}

\begin{ackno}
 \rm 

 K.C.~and G.L.~were supported by The Maxwell Institute Graduate School in Analysis and its Applications, a Centre for Doctoral Training funded by the UK Engineering and Physical Sciences Research Council (Grant EP/L016508/01), the Scottish Funding Council, Heriot-Watt University and the University of Edinburgh.
 K.C.~and G.L.~also acknowledge support
 from the European Research Council (grant no.~637995 ``ProbDynDispEq'').
T.O.~was supported by the European Research Council (grant no.~637995 ``ProbDynDispEq''
and grant no.~864138 ``SingStochDispDyn").
The authors would like to thank the anonymous referee for helpful comments.

\end{ackno}

\end{document}